\documentclass[onecolumn,onefignum,final]{siamltex}
\usepackage[margin=1.206in]{geometry} 

\usepackage[latin1]{inputenc}
\usepackage[boxed]{algorithm2e}
\usepackage{amsmath}
\usepackage{amssymb}
\usepackage{mathrsfs}
\usepackage{graphicx}
\usepackage{subfigure}
\usepackage{multirow}
\usepackage{enumerate}
\usepackage{amscd}

\DeclareMathOperator{\trace}{trace}

\newcommand{\R}{\mathbb{R}}\newcommand{\C}{\mathbb{C}}

\newtheorem{remark}[theorem]{Remark}
\newtheorem{example}[theorem]{Example}

\title{Quasi-Linear Compressed Sensing
        }

\author{Martin Ehler\thanks{University of Vienna,
Faculty of Mathematics, 
Oskar-Morgenstern-Platz 1
A-1090 Vienna, ({\tt martin.ehler@univie.ac.at}).}
        \and Massimo Fornasier\thanks{Technische Universit\"at M\"unchen, 
Fakult\"at f\"ur Mathematik, 
Boltzmannstr.~3, 
85748 Garching bei M\"unchen, ({\tt massimo.fornasier@ma.tum.de})}
\and Juliane Sigl\thanks{Technische Universit\"at M\"unchen, 
Fakult\"at f\"ur Mathematik, 
Boltzmannstr.~3, 
85748 Garching bei M\"unchen, ({\tt sigl.juliane@gmail.com})} 
}

\begin{document}

\maketitle

\begin{abstract}
Inspired by  significant real-life applications, in particular, sparse phase retrieval and sparse pulsation frequency detection in Asteroseismology, we investigate a general framework for {\it compressed sensing}, where the measurements
are {\it quasi-linear}. We formulate natural generalizations of the well-known {\it Restricted Isometry Property} (RIP) towards nonlinear measurements, which allow us to prove both unique identifiability of sparse signals as well as
the convergence of recovery algorithms to compute them efficiently. We show that for certain randomized quasi-linear measurements, including Lipschitz perturbations of classical RIP matrices and phase retrieval from random projections, the
proposed restricted isometry properties hold with high probability.  
We analyze  a generalized {\it Orthogonal Least Squares} (OLS) under the assumption that magnitudes of signal entries to be recovered decay fast.  Greed is  good again, as we show that this algorithm performs efficiently in phase retrieval and Asteroseismology. For situations where the decay assumption on the signal does
not necessarily hold, we propose two alternative algorithms, which are natural generalizations of the well-known {\it iterative hard and soft-thresholding}. While these algorithms are rarely successful for the mentioned applications, we show their strong
recovery guarantees for quasi-linear measurements which are Lipschitz perturbations of RIP matrices.
\end{abstract}



\pagestyle{myheadings}
\thispagestyle{plain}
\markboth{M.~Ehler, M.~Fornasier, J.~Sigl}{Quasi-Linear Compressed Sensing}

\section{Introduction}
Compressed sensing addresses the problem of recovering nearly-sparse signals from vastly incomplete measurements ~\cite{Candes:2006ac,Candes:2006aa,Candes:2005vn,Candes:2006ab,Donoho:2006ab}.  By using the prior assumptions on the signal, the number of measurements can be well below the Shannon sampling rate and effective reconstruction algorithms are available. The standard compressed sensing approach deals with {\it linear} measurements. The success of signal recovery algorithms often relies on the so-called {\it Restricted Isometry Property} (RIP) \cite{Candes:2006aa,Candes:2006ab,Fornasier:2008aa,Pfander:fk,Rauhut:2012fk,Rauhut:2008aa}, which is a near-identity spectral property of small submatrices of the measurement Gramian. The RIP condition is satisfied with high probability and nearly optimal number of measurements for a large class of random measurements \cite{Bah:2010uq,Baraniuk:2008fk,Candes:2005vn,Pfander:fk,Rauhut:2012fk}, which explains the popularity of all sorts of random sensing approaches. The most effective recovery algorithms are based either on a greedy approach or on variational models, such as $\ell_1$-norm minimization, leading to suitable iterative thresholded gradient descent methods. In the literature of mathematical signal processing, greedy algorithms for sparse recovery originate from the so-called Matching Pursuit \cite{Mallat:1993aa}, although several predecessors were well-known in other communities. Among astronomers and asteroseismologists, for instance, {\it Orthogonal Least Squares} (OLS) \cite{Gribonval:2011fk} was already in use in the '60s for the detection of significant frequencies of star light-spectra (the so-called pre-whitening) \cite{Barning:1963uq}. We refer to  \cite{Needell:2009kx,Temlyakov:2011fk} for more recent developments on greedy approaches. Iterative thresholding algorithms have instead a variational nature and they are designed to minimize the discrepancy with respect
to the measurements and simultaneously to promote sparsity by iterated thresholding operations. We refer to \cite{Blumensath:2009aa,Daubechies:2004aa,Fornasier:2008kx} and references therein for more details on such iterative schemes for sparse recovery
from linear measurements.

Often models of physical measurements in the applied sciences and engineering, however, are  not linear and it is of utmost interest to investigate to which extent the theory of compressed sensing can be generalized to nonlinear measurements. {  Two relevant real-life applications in physics can be mentioned,} asteroseismic light measurements \cite{Aerts:2010fk} to determine the shape of pulsating stars and magnitude measurements in phase retrieval problems important to diffraction imaging and X-ray crystallography \cite{Drenth:2010fk,Fienup:1982vn,Gerchberg:1972kx}.  
There are already several recent attempts towards nonlinear compressed sensing, for instance the work by Blumensath et al.~\cite{Blumensath:2012fk,Blumensath:2008fk}, quadratic measurements are considered in \cite{Li:2012uq}, and further nonlinear inverse problems are analyzed in \cite{Ramlau:2006fk,Ramlau:2007zr}.  Phase retrieval is an active field of research nowadays and has been addressed by related approaches \cite{Bachoc:2012fk,Bauschke:2002ys,Candes:uq,Eldar:2012fk,Seifert:2006fk}. 
\\

In the present paper we provide a more unified view, by restricting the possible nonlinearity of the measurements to quasi-linear maps, which are sufficiently smooth, at least Lipschitz, and they fulfill generalized versions of
the classical RIP. In contrast to the situation of linear measurements, the nonlinearity of the measurements actually {  plays in a differing manner} within different recovery algorithms. Therefore it is necessary to design corresponding forms of RIP depending on the recovery strategies used, see conditions \eqref{eq:1}, \eqref{eq:1b} for a greedy algorithm and \eqref{eq:RIP again}, \eqref{lastRIP} for iterative thresholding algorithms. 
In particular, we show that for certain randomized quasi-linear measurements, including Lipschitz perturbations of classical RIP matrices and phase retrieval from random projections, the
proposed restricted isometry properties hold with high probability. While for the phase retrieval problem the stability results in \cite{Eldar:2012fk} are restricted to the real setting, we additionally extend them to the complex case.

Algorithmically we first focus on a generalized {\it Orthogonal Least Squares} (OLS).
Such a greedy approach was already proposed in \cite{Blumensath:2008fk}, although there no analysis of convergence was yet provided. We show within the framework of quasi-linear compressed sensing problems the recovery guarantees of this algorithm, by taking inspiration from \cite{Davenport:2010fk}, where a similar analysis is performed for linear measurements. 
The greedy algorithm we propose works for both types of applied problems mentioned above, i.e., Asteroseismology and phase retrieval. 
Let us stress that for the latter and for signals which have rapidly decaying nonincreasing rearrangement, few iterations of this greedy algorithm are sufficient to obtain a good recovery accuracy. Hence, our approach seems very competitive compared to the semi-definite program used in \cite{Candes:uq} for phase retrieval, by recasting the problem into a demanding optimization on matrices.

The greedy strategy as derived here, however, also inherits two drawbacks: (1) the original signal is required to satisfy the mentioned decay conditions, and (2) the approach needs careful implementations of multivariate global optimization to derive high accuracy for signal recovery.  

To possibly circumvent those drawbacks, we then explore alternative strategies, generalizing iterative hard- and soft-thresholding methods, which allow us to recover nearly-sparse signals not satisfying the decay assumptions. The results we present for hard-thresholding are mainly based on Blumensath's findings in \cite{Blumensath:2012fk}. For iterative soft-thresholding, we prove in an original fashion the convergence of the algorithm towards a limit point and we  bound its distance to the original signal.
While iterative thresholding algorithms are rarely successful for phase retrieval problems, we show their strong recovery guarantees for quasi-linear measurements which are Lipschitz perturbations of RIP matrices. We further emphasize in our numerical experiments that different iterative algorithms based on contractive principles do provide rather diverse success recovery results for the same problem, especially when nonlinearities are involved: this is due to the fact that the basins of attraction towards fixed points of the corresponding iterations can be significantly different. In our view, this is certainly sufficient motivation to explore {  several algorithmic approaches and not  restricting ourselves just to a favorite one.
}
\\

As we clarified above, each algorithmic approach requires a different treatment of the nonlinearity of the measurements with the consequent need of defining corresponding generalizations of the RIP. Hence, we develop the presentation of our results
according to the different algorithms, starting first with the generalized Orthogonal Least Squares, and later continuing with the iterative thresholding algorithms. Along the way, we present examples of applications and we show how to fulfill the required deterministic conditions of convergence by randomized quasi-linear measurements. The outline of the paper is as follows: In Section \ref{sec:2}, we introduce the nonlinear compressed sensing problem, and in Section 3 we derive a greedy scheme for nearly sparse signal reconstruction. We show applications of this algorithm in Section \ref{sec:quasi linear} to analyze simulated asteroseismic data towards the detection of frequency pulsation of stars. In Section \ref{sec:phase}, we discuss refinements and changes needed for the phase retrieval problem and also provide numerical experiments. For signals not satisfying the decay assumptions needed for the greedy algorithm to converge, iterative thresholding algorithms are discussed in Section \ref{sec:thresholding}. 

\section{Quasi-linear compressed sensing}\label{sec:2}
\subsection{The nonlinear model}
In the by now classical compressed sensing framework, an unknown  nearly sparse signal $\hat{x}\in\R^d$ is to be reconstructed from $n$ linear measurements, with $n \ll d$, and one models this setting as the solution of a linear system
\begin{equation*}
A\hat{x}+e = b,
\end{equation*}
where $e$ is some noise term and the $i$-th row of $A\in\R^{n\times d}$ corresponds to the $i$-th linear measurement on the unknown signal $\hat x$ with outcome $b_i$.  We say that $A$ satisfies the Restricted Isometry Property (RIP) of order $k$ with $0<\delta_k<1$ if 
\begin{equation}\label{eq:RIP}
 (1-\delta_k)\|x\|\leq \|Ax\|\leq (1+\delta_k)\|x\|,
\end{equation}
for all $x \in \R^d$ with at most $k$ nonzero entries. We call such vectors $k$-sparse. If  $A$ satisfies the RIP of order $2k$  and $\delta_{2k} < \frac{2}{3+ \sqrt{7/4}} \approx 0.4627$, then signal recovery is possible up to noise level and $k$-term approximation error.
It should be mentioned that large classes of random matrices $A\in\R^{n\times d}$ satisfy the RIP with high probability for the (nearly-)optimal dimensionality scaling $k = O\left ( \frac{n}{1 + \log(d/n)^\alpha} \right)$. We refer to  ~\cite{Baraniuk:2008fk,Candes:2006aa,Candes:2005vn,Candes:2006ab,Donoho:2006ab,Pfander:fk,Rauhut:2012fk}
 for the early results and \cite{Foucart:2013vn} for a recent extended treatise.

Many real-life applications in physics and biomedical sciences, however, carry some strongly nonlinear structure,  so that the linear model is not suited anymore, even as an approximation. Towards the definition of a nonlinear framework for compressed sensing, we shall consider for $n \ll d$ a map $A:\R^d\rightarrow \R^{n}$, which is not anymore necessarily linear, and aim at reconstructing $\hat{x}\in\R^d$ from the measurements $b\in\R^n$ given by 
\begin{equation}\label{eq:quasi linear}
A(\hat{x})+e = b.
\end{equation}
Similarly to linear problems, also the unique and stable solution of the equation \eqref{eq:quasi linear} is in general an impossible task, unless we require certain a priori assumptions on $\hat x$, and some stability properties similar to \eqref{eq:1} for the nonlinear map $A$. As the variety of possible nonlinearities is extremely vast, it is perhaps too ambitious to expect that generalized RIP properties can be verified for any type of nonlinearity. As a matter of fact, and as we shall show in details below, most of the nonlinear models with stability properties which allow for nearly sparse signal recovery, have a smooth quasi-linear nature. With this we mean that there exists a Lipschitz map $F:\R^d \to \R^{n\times d}$ such that $A(x) = F(x) x$, for all $x \in \R^d$. However, in the following we will use and explicitly highlight this quasi-linear structure only when necessary.
\begin{algorithm}
\KwSty{$\ell_p$-greedy algorithm}:\\
\KwIn{$A:\R^d\rightarrow \R^{n}$, $b\in\R^n$ 
}
Initialize $x^{(0)}=0\in\R^d$, $\Lambda^{(0)}=\emptyset$\\
\For{$j=1,2,\ldots$ until some stopping criterion is met}{
\For{$l\not\in\Lambda^{(j-1)}$}{
$\Lambda^{(j-1,l)}:=\Lambda^{(j-1)}\cup\{l\}$
\begin{equation*}
x^{(j,l)}:=\arg\min_{\{x:\supp(x)\subset\Lambda^{(j-1,l)}\}} \big\|A(x)-b   \big\|_{\ell_p}
\end{equation*}
}
Find index that minimizes the error:
\begin{equation*}
l_j:=\arg\min_{l} \big\|A(x^{(j,l)})-b   \big\|_{\ell_p}
\end{equation*}
Update: $x^{(j)}:=x^{(j,l_j)}$, $\Lambda^{(j)}:=\Lambda^{(j-1,l_j)}$
}
\KwOut{$x^{(1)}$, $x^{(2)},\ldots$}
\caption{The $\ell_p$-greedy algorithm terminates after finitely many steps, but we need to solve $d-j$ many $j$-dimensional optimization problem in the $j$-th step. If we know that $b=A(\hat{x})+e$ holds and $\hat{x}$ is $k$-sparse, then the stopping criterion $j\leq k$ appears natural, but can also be replaced with other conditions.}\label{algo:1}
\end{algorithm}

Our first approach towards the solution of \eqref{eq:quasi linear} will be based on a greedy principle, since it is also perhaps the most intuitive one: we search first for the best $1$-sparse signal which is minimizing the discrepancy with respect to the measurements and then we seek for a next best matching $2$-sparse signal having as one of the active entries the one previously detected, and so on. This method is formally summarized in the $\ell_p$-norm matching greedy Algorithm \ref{algo:1}.
For the sake of clarity, we mention that $\|x \|$ denotes the standard Euclidean norm of any vector $x \in \mathbb R^d$, while $\|x \|_{\ell_p} = \left ( \sum_{i=1}^d |x_i|^p \right )^{1/p}$ is its $\ell_p$-norm for $1 \leq p < \infty$.
Moreover, when dealing with matrices, we denote with $\|A\|=\|A\|_2$ the spectral norm of the matrix $A$ and with $\|A\|_{HS}$ its Hilbert-Schmidt norm.

\section{Greed is good - again}\label{sec:3}
\subsection{Deterministic conditions I}


%
%

Greedy algorithms have already proven useful and efficient for many sparse signal reconstruction problems in a linear setting, cf.~\cite{Tropp:2004vn}, and we refer to \cite{Blanchard:2011fk} for a more recent treatise. Before we can state our reconstruction result here, we still need some preparation.  
The  nonincreasing rearrangement of $x\in\R^d$ is defined as
\begin{equation*}
r(x)=(|x_{j_1}|,\ldots,|x_{j_d}|)^\top,\quad\text{where}\quad  |x_{j_i}|\geq |x_{j_{i+1}}|,\text{ for } i=1,\ldots,d-1.
\end{equation*}
For $0<\kappa<1$, we define the class of $\kappa$-rapidly decaying vectors in $\R^d$ by
  \begin{equation*}
  \mathcal{D}_\kappa=\{x\in\R^d: r_{j+1}(x)\leq \kappa r_j(x),\text{ for $j=1,\ldots,d-1$} \}.
  \end{equation*}
Given $x\in\R^d$, the vector $x_{\{j\}}\in\R^d$  is the best $j$-sparse approximation of $x$, i.e., it consists of the $j$ largest entries of $x$  in absolute value and zeros elsewhere. Signal recovery is possible under decay and stability conditions using the $\ell_p$-greedy Algorithm \ref{algo:1}, which is a generalized Orthogonal Least Squares \cite{Gribonval:2011fk}:

\begin{theorem}\label{th:rec result 1}
Let $b=A(\hat{x})+e$, where $\hat{x}\in\R^d$ is the signal to be recovered and $e\in\R^n$ is a noise term. Suppose further that $1\leq k\leq d$, $r_k(\hat{x})\neq 0$, and $1\leq p<\infty$. 
If the following conditions hold,
\begin{itemize}
\item[(i) ] there are $\alpha_k,\beta_k>0$ such that, for all $k$-sparse $y\in\R^d$,
\begin{equation}\label{eq:1}
\alpha_k \|\hat{x}_{\{k\}}-y\|\leq \|A(\hat{x}_{\{k\}})-A(y)\|_{\ell_p}\leq \beta_k \|\hat{x}_{\{k\}}-y\|,
\end{equation}
\item[(ii) ]  $\hat{x}\in\mathcal{D}_\kappa$ such that  $\kappa < \frac{\tilde{\alpha}_k}{\sqrt{\tilde{\alpha}_k^2+(\beta_k+2L_k)^2}}$, where $0<\tilde{\alpha}_k\leq \alpha_k-2\|e\|_{\ell_p}/r_k(\hat{x})$ and $L_k\geq 0$ with $\|A(\hat{x})-A(\hat{x}_{\{k\}})\|_{\ell_p}\leq L_k \|\hat{x}-\hat{x}_{\{k\}}\|$,
\end{itemize}
then the $\ell_p$-greedy Algorithm \ref{algo:1} yields a sequence $(x^{(j)})_{j=1}^k$ satisfying $\supp(x^{(j)}) = \supp(\hat{x}_{\{j\}})$ and 
\begin{equation*}
\| x^{(j)} -\hat{x}\| \leq \|e\|_{\ell_p}/\alpha_k+\kappa^jr_1(\hat{x}) \sqrt{2}\left (1+\frac{\beta_k+2L_k}{\alpha_k} \right ).
\end{equation*}
If $\hat{x}$ is $k$-sparse, then $\|x^{(k)}-\hat{x}\|\leq \|e\|_{\ell_p}/\alpha_k$.
\end{theorem}

According to $0<\tilde{\alpha}_k\leq \alpha_k-2\|e\|_{\ell_p}/r_k(\hat{x})$, the noise term $e$ must be small and we implicitly suppose that $r_k(\hat{x})\neq 0$. Otherwise, we can simply choose a smaller $k$. 
If the signal $\hat{x}$ is $k$-sparse and the noise term $e$ equals zero, then the $\ell_p$-greedy Algorithm \ref{algo:1} in Theorem \ref{th:rec result 1} yields $x^{(k)}=\hat{x}$. 

\begin{remark}
A similar greedy algorithm was proposed in \cite{Blumensath:2008fk} for nonlinear problems, and our main contribution here is the careful analysis of its signal recovery capabilities. 
Conditions of the type \eqref{eq:1} have also been used in \cite{Blumensath:2012fk}, but with additional restrictions, so that the constants $\alpha_k$ and $\beta_k$ must be close to each other. In our Theorem \ref{th:rec result 1}, we do not need any constraints on such constants, because the decay conditions on the signal can compensate for this. A similar relaxation using decaying signals was proposed in \cite{Davenport:2010fk} for linear operators $A$, but even there the authors still assume $\beta_k/\alpha_k<2$.  We do not require here any of such conditions.
\end{remark}


The proof of Theorem \ref{th:rec result 1} extends the preliminary results obtained in \cite{Sigl:2013fk}, which we split and generalize 
in the following two lemmas:
\begin{lemma}\label{lemma:-1}
If $\hat{x}\in\R^d$ is contained in $\mathcal{D}_\kappa$, then 
\begin{equation}\label{eq:2ba0}
\|\hat{x}-\hat{x}_{\{j\}}\| < r_{j+1}(\hat{x}) \frac{1}{\sqrt{1-\kappa^2}} \leq r_{j}(\hat{x}) \frac{\kappa}{\sqrt{1-\kappa^2}}.
\end{equation}
\end{lemma}

The proof of Lemma \ref{lemma:-1} is a straightforward calculation, in which the geometric series is used, so we omit the details. 
\begin{lemma}\label{lemma:0}
Fix $1\leq k\leq d$ and suppose that $r_k(\hat{x})\neq 0$. If $\hat{x}\in\R^d$ is contained in $\mathcal{D}_\kappa$ with $\kappa < \frac{\tilde{\alpha}}{\sqrt{\tilde{\alpha}^2+(\beta+2L)^2}}$, where $0<\tilde{\alpha}\leq \alpha-2\|e\|_{\ell_p}/r_k(\hat{x})$, for some $\alpha,\beta,L>0$, then, for $j=1,\ldots,k$,
\begin{equation}\label{eq:2ba}
\alpha r_j(\hat{x}) > 2\|e\|_{\ell_p}+2L\|\hat{x}-\hat{x}_{\{k\}}\| +\beta \|\hat{x}_{\{j\}}-\hat{x}_{\{k\}}\|.
\end{equation}
\end{lemma}

\begin{proof}
It is sufficient to consider $\hat{x}$, which are not $j$-sparse. A short calculation reveals that the condition on $\kappa$ implies
\begin{equation*}
\alpha-2\|e\|_{\ell_p}/r_k(\hat{x})-\frac{\kappa}{\sqrt{1-\kappa^2}} (\beta+2L)>0.
\end{equation*}
We multiply the above inequality by $r_j(\hat{x})$, so that $r_j(\hat{x})\geq r_k(\hat{x})$ yields
\begin{equation*}
\alpha r_j(\hat{x})-2\|e\|_{\ell_p}-2L\frac{\kappa}{\sqrt{1-\kappa^2}} r_k(\hat{x})-\beta\frac{\kappa}{\sqrt{1-\kappa^2}}r_j(\hat{x})>0.
\end{equation*}
Lemma \ref{lemma:-1} now implies \eqref{eq:2ba}.
 \end{proof}
 
We can now turn to the proof of Theorem \ref{th:rec result 1}:

\begin{proof}[Proof of Theorem \ref{th:rec result 1}]
We must check that the index set selected by the $\ell_p$-greedy algorithm matches the location of the nonzero entries of $\hat{x}_{\{k\}}$. We use induction and observe that nothing needs to be checked for $j=0$. In the induction step, we suppose that $\Lambda^{(j-1)}\subset \supp(\hat{x}_{\{j-1\}})$. Let us choose $l\not\in \supp(\hat{x}_{\{j\}})$. The lower inequality in \eqref{eq:1} yields
\begin{align*}
\|b-A(x^{(j,l)})\|_{\ell_p} &\geq  -\|e\|_{\ell_p}+ \|A(\hat{x})-A(x^{(j,l)})\|_{\ell_p}\\
& \geq -\|e\|_{\ell_p}  -   \|A(\hat{x})-A(\hat{x}_{\{k\}})\|_{\ell_p}+\|A(\hat{x}_{\{k\}})-A(x^{(j,l)})\|_{\ell_p}	\\
& \geq -\|e\|_{\ell_p} -L_k \|\hat{x}-\hat{x}_{\{k\}} \|+ \alpha_k \| \hat{x}_{\{k\}}-x^{(j,l)}\|  \\
& \geq -\|e\|_{\ell_p} -L_k \|\hat{x}-\hat{x}_{\{k\}} \|+ \alpha_k  r_j(\hat{x}),
\end{align*}
where the last inequality is a crude estimate based on $l\not\in \supp(\hat{x}_{\{j\}})$. Thus, Lemma \ref{lemma:0}  implies 
\begin{equation}\label{eq:op 1}
\|b-A(x^{(j,l)})\|_{\ell_p} > \|e\|_{\ell_p} +(\beta_k+L_k)\|\hat{x}-\hat{x}_{\{k\}}\|+\beta_k \|\hat{x}-\hat{x}_{\{j\}}\|.
\end{equation}
On the other hand, the minimizing property of $x^{(j)}$ and the upper inequality in \eqref{eq:1} yield
\begin{align*}
\|b-A(x^{(j)})\|_{\ell_p}  &\leq \|b-A(\hat{x}_{\{j\}})\|_{\ell_p} \\
& \leq \|e\|_{\ell_p}+\|A(\hat{x})-A(\hat{x}_{\{j\}})\|_{\ell_p}\\
&\leq \|e\|_{\ell_p}+\|A(\hat{x})-A(\hat{x}_{\{k\}})\|_{\ell_p}+\|A(\hat{x}_{\{k\}})-A(\hat{x}_{\{j\}})\|_{\ell_p}\\
& \leq\|e\|_{\ell_p}+L_k \|\hat{x}-\hat{x}_{\{k\}}\|+\beta_k\|\hat{x}_{\{k\}}-\hat{x}_{\{j\}}\|.
\end{align*}
The last line and \eqref{eq:op 1} are contradictory if $x^{(j)}=x^{(j,l)}$, so that we must have $x^{(j)}=x^{(j,l)}$, for some $l\in \supp(\hat{x}_{\{j\}})$, which concludes the part about the support. 

Next, we shall derive the error bound. Standard computations yield
\begin{align*}
\| x^{(j)} -\hat{x}\|  & \leq \| x^{(j)} -\hat{x}_{\{k\}}\| +\|\hat{x}_{\{k\}} -\hat{x}\| \\
&\leq 1/\alpha_k  \|A(x^{(j)})-A(\hat{x}_{\{k\}})\|_{\ell_p}+\|\hat{x}_{\{k\}} -\hat{x}\| \\
&\leq 1/\alpha_k  \|A(x^{(j)})-A(\hat{x})\|_{\ell_p}+1/\alpha_k\|A(\hat{x})-A(\hat{x}_{\{k\}})\|_{\ell_p}+\|\hat{x}_{\{k\}} -\hat{x}\| \\
&\leq 1/\alpha_k  \|A(\hat{x}_{\{j\}})-A(\hat{x})\|_{\ell_p}+\|e\|_{\ell_p}/\alpha_k+L_k/\alpha_k\|\hat{x}-\hat{x}_{\{k\}}\|+\|\hat{x}_{\{k\}} -\hat{x}\| \\
&\leq 1/\alpha_k  \|A(\hat{x}_{\{j\}})-A(\hat{x}_{\{k\}})\|_{\ell_p}+\|e\|_{\ell_p}/\alpha_k+1/\alpha_k  \|A(\hat{x}_{\{k\}})-A(\hat{x})\|_{\ell_p}\\
& \qquad+L_k/\alpha_k\|\hat{x}-\hat{x}_{\{k\}}\|+\|\hat{x}_{\{k\}} -\hat{x}\| \\
&\leq \beta_k/\alpha_k  \|\hat{x}_{\{j\}}-\hat{x}_{\{k\}}\|+\|e\|_{\ell_p}/\alpha_k+2L_k/\alpha_k\|\hat{x}-\hat{x}_{\{k\}}\|+\|\hat{x}_{\{k\}} -\hat{x}\| \\
& \leq \big(\beta_k/\alpha_k+2L_k/\alpha_k+1  \big)\|\hat{x}_{\{k\}} -\hat{x}\|+\|e\|_{\ell_p}/\alpha_k\\
& \leq \big(\beta_k/\alpha_k+2L_k/\alpha_k+1  \big)r_{j+1}(\hat{x}) \frac{1}{\sqrt{1-\kappa^2}}+\|e\|_{\ell_p}/\alpha_k\\
& \leq \big(\beta_k/\alpha_k+2L_k/\alpha_k+1  \big) \kappa^jr_1(\hat{x}) \frac{1}{\sqrt{1-\kappa^2}}+\|e\|_{\ell_p}/\alpha_k.
\end{align*}
Few rather rough estimates yield 
\begin{equation*}
\frac{\beta_k/\alpha_k+2L_k/\alpha_k+1 }{\sqrt{1-\kappa^2}}\leq \sqrt{2}(1+\frac{\beta_k+2L_k}{\alpha_k})
\end{equation*}
which concludes the proof. 
\end{proof}

\subsection{Examples of inspiring applications}

In this subsection we present examples where Algorithm \ref{algo:1} can be successfully used. We start with an abstract example of a nonlinear Lipschitz perturbation of a linear model
and then we consider a relevant real-life example from Asteroseismology.

\subsubsection{Lipschitz perturbation of a RIP matrix}
As an explicit example of $A$ matching the requirements of Theorem \ref{th:rec result 1} with high probability, we propose Lipschitz perturbations of RIP matrices:
\begin{proposition}\label{th:example}
If $A$ is chosen  as
\begin{equation}
A(x):=A_1x+\epsilon f(\|x-x_0\|)A_2x, \label{firstqn}
\end{equation}
where $A_1\in\R^{n\times d}$ satisfies the RIP \eqref{eq:RIP} of order $k$ and constant $0 < \delta_k < 1$, $x_0\in\R^d$ is some reference vector, $f:[0,\infty)\rightarrow\R$ is a  bounded  Lipschitz continuous function, $\epsilon$ is a sufficiently small scaling factor, and $A_2\in\R^{n\times d}$ arbitrarily fixed, then there are constants $\alpha_k,\beta_k>0$, such that the assumptions in Theorem \ref{th:rec result 1} hold for $p=2$. 
\end{proposition}

\begin{proof}
We first check on $\beta_k$. If $L$ denotes the Lipschitz constant of $f$, then we obtain
\begin{align*}
\|A(\hat{x})-A(y)\| & =  \| A_1\hat{x} - A_1y+  \epsilon f(\|\hat{x}-x_0\|)A_2\hat{x} - \epsilon f(\|y-x_0\|)   A_2y \|\\
&  = \|A_1\hat{x} - A_1y+\epsilon \big[f(\|\hat{x}-x_0\|)A_2\hat{x} -f(\|y-x_0\|)   A_2\hat{x}\\
& \qquad +f(\|y-x_0\|)   A_2\hat{x}- f(\|y-x_0\|)   A_2y\big] \|\\
& \leq  (1+\delta_k)\|\hat{x}-y\|+ \epsilon L  \big|\|\hat{x}-x_0\| - \|y-x_0\|\big|\|A_2 \|_{2}\|\hat{x}\|+\epsilon B\|A_2\|_{2} \|\hat{x}-y\|\\
& \leq (1+\delta_k+\epsilon L   \|A_2 \|_{2}\|\hat{x}\| +\epsilon B\|A_2\|_2)\|\hat{x}-y\|,
\end{align*} 
where we have used the reverse triangular inequality and $B=\sup\{|f(\|x-x_0\|)|: x\in\R^d\;\text{is $k$-sparse}\}$. Thus, we can choose $\beta_k:=1+\delta_k+\epsilon L   \|A_2 \|_{2}r +B\|A_2\|_2$, where $r \geq  \| \hat{x}\|$.

Next, we derive a suitable $\alpha_k$. For $k$-sparse $y\in\R^d$, we derive similarly to the above calculations
\begin{align*}
 \|A(\hat{x})-A(y)\|& \geq \| A_1\hat{x}-A_1y\| - \epsilon\|f(\|\hat{x}-x_0\|) A_2\hat{x} -  f(\|y-x_0\|) A_2y\|\\
& \geq (1-\delta_{k})\|\hat{x}-y\| - \epsilon L\|A_2\|_2\|\hat{x}\| \|\hat{x}-y\| - \epsilon B\|A_2\|_2\|\hat{x}-y\|\\
& = (1-\delta_{k}-\epsilon\|A_2\|_2(L\|\hat{x}\|+B)\|\hat{x}-y\|.
\end{align*}
If $\epsilon$ is sufficiently small, we can choose $\alpha_k:=1-\delta_{k}-\epsilon\|A_2\|_2(L r +B)$.
\end{proof}

Any matrix satisfying the RIP of order $k$ with high probability, {  for instance being within certain classes of random matrices \cite{Foucart:2013vn}, induces} maps $A$ via Proposition \ref{th:example} that satisfy the assumptions of Theorem \ref{th:rec result 1}. {  Notice that the form of nonlinearity considered in \eqref{firstqn} is actually quasi-linear, i.e., $A(x) = F(x) x$, where $F(x)  =A_1 +\epsilon f(\|x-x_0\|)A_2$.}

\subsubsection{Quasi-linear compressed sensing in Asteroseismology}\label{sec:quasi linear}

Asteroseismology studies the oscillation occurring inside variable pulsating stars as
seismic waves \cite{Aerts:2010fk}. Some regions of the stellar surface contract and heat up
while others expand and cool down in a regular pattern causing observable changes in the light intensity.  This also means that areas of
different temperature correspond to locations of different expansion of the star and characterize its shape.
Through the analysis of the frequency spectra it is possible
to determine the internal stellar structure. Often complex pulsation patterns with multiperiodic oscillations
are observed and their identification is needed. 
\roman{figure} 
\begin{figure}
\centering
\subfigure[Original shape by means of $u$ and the corresponding $2$-sparse Fourier coefficients $x$.]{
\includegraphics[width=.33\textwidth]{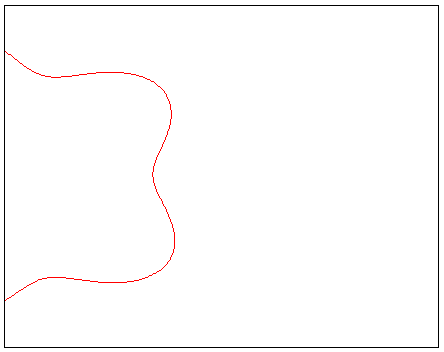}\includegraphics[width=.33\textwidth]{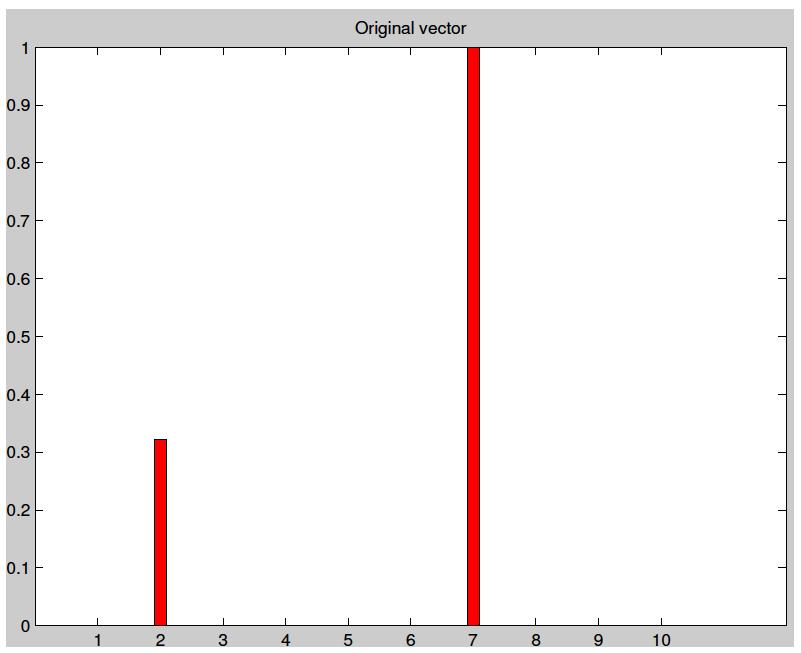}}

\subfigure[$1$- and $2$-sparse reconstruction of the Fourier coefficients.]{
\includegraphics[width=.33\textwidth]{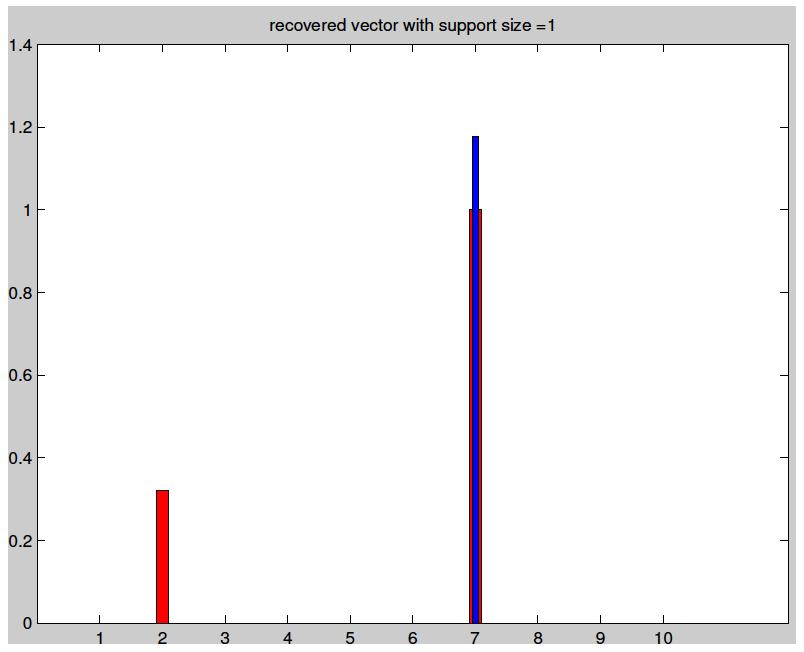}\includegraphics[width=.33\textwidth]{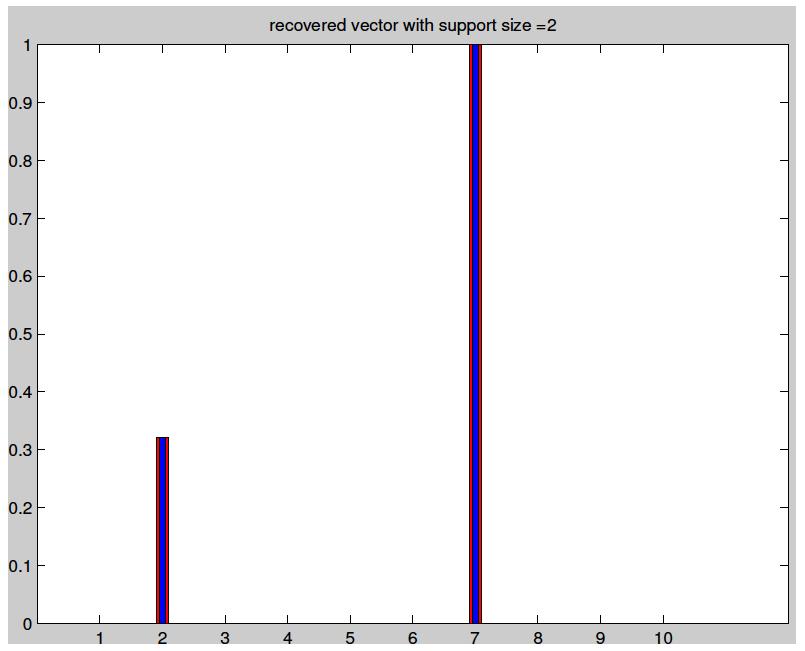}}
\caption{Original $2$-sparse signal and reconstructed pulsation patterns. Red corresponds to the original signal, the reconstruction is given in blue.}\label{fig:pulsation}
\end{figure}

\begin{figure}
\centering
\subfigure[Original shape by means of $u$ and the corresponding $3$-sparse Fourier coefficients $x$.]{
\includegraphics[width=.33\textwidth]{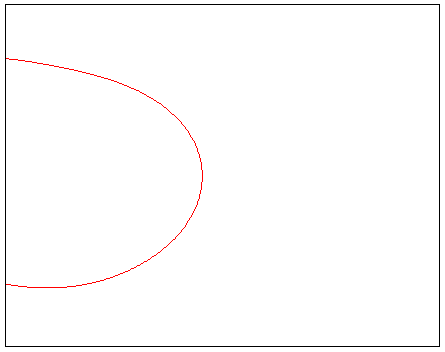}\includegraphics[width=.33\textwidth]{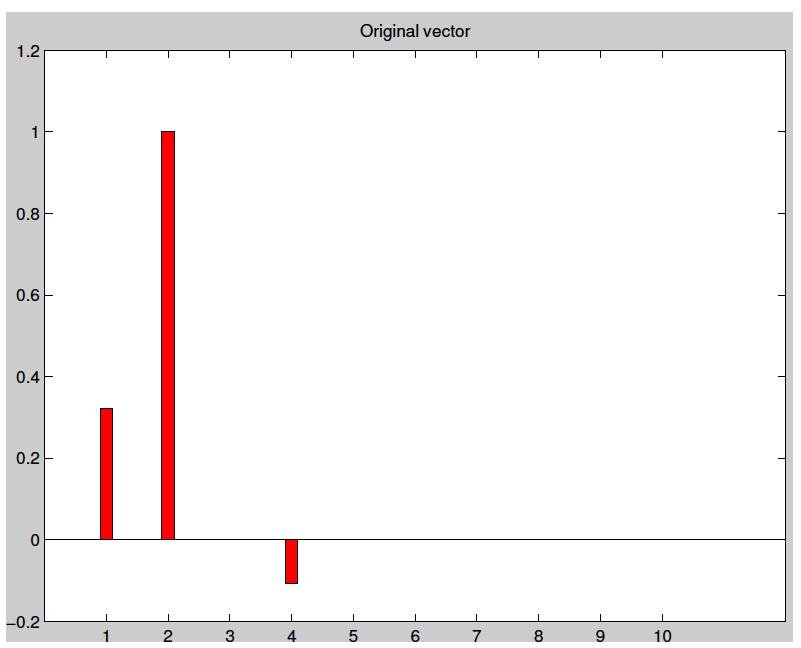}}

\subfigure[$1$-, $2$-, and $3$-sparse reconstruction.]{
\includegraphics[width=.33\textwidth]{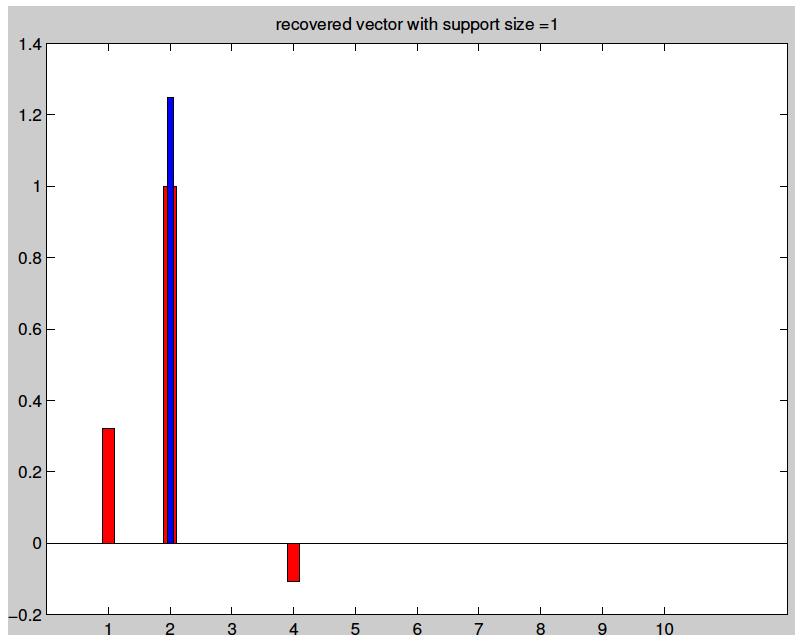}\includegraphics[width=.33\textwidth]{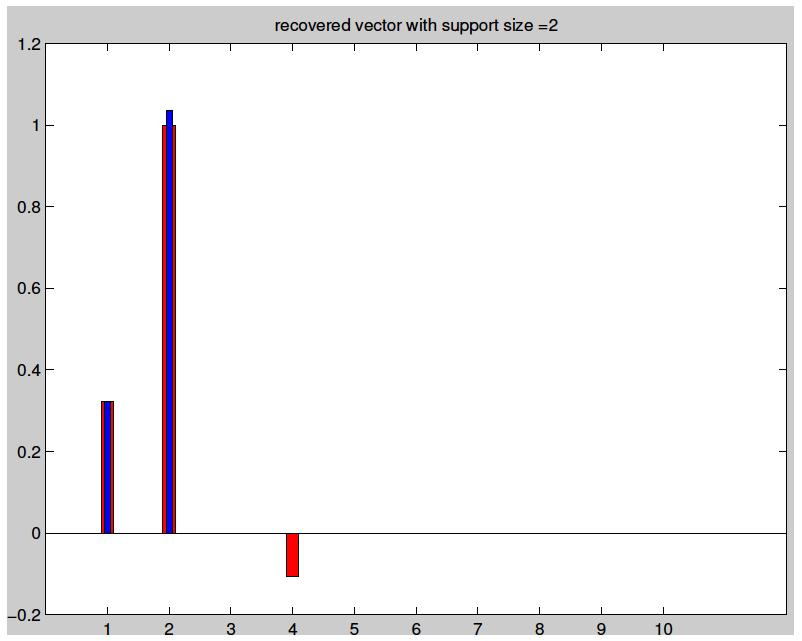}\includegraphics[width=.33\textwidth]{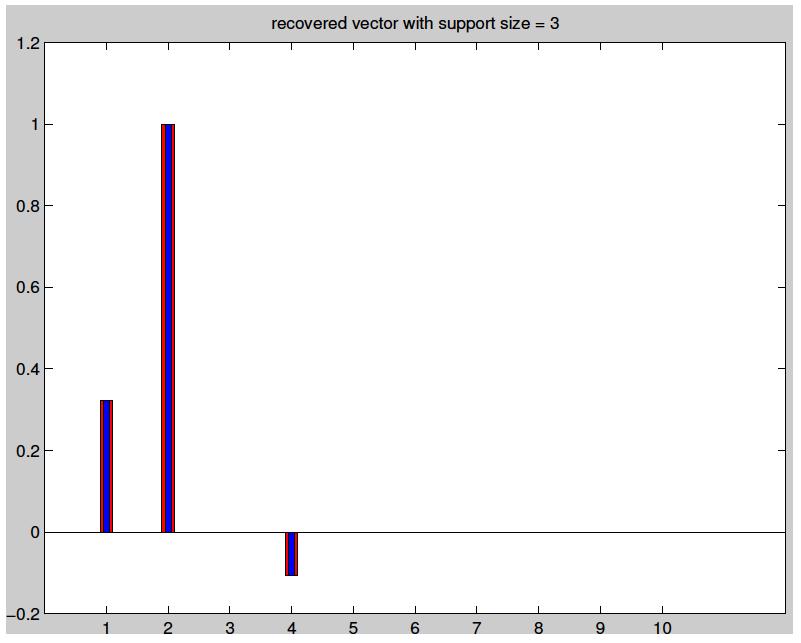}}
\caption{Original $3$-sparse signal and reconstructed pulsation patterns. Red corresponds to the original signal, the reconstruction is given in blue.}\label{fig:pulsation 2}
\end{figure}

\begin{figure}
\centering
\subfigure[original shape]{
\includegraphics[width = .3\textwidth]{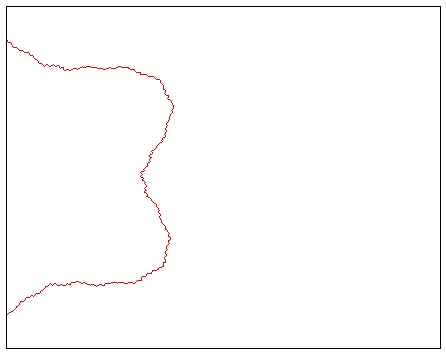}}
\subfigure[reconstructed shape]{
\includegraphics[width = .3\textwidth]{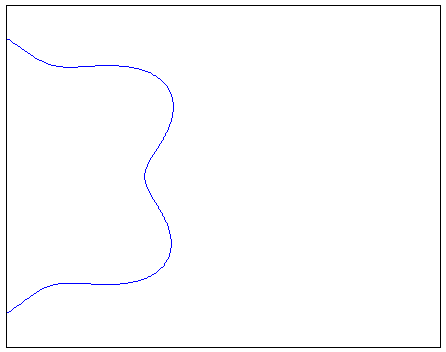}}
\caption{The original signal is rapidly decaying with higher frequencies and we reconstruct $3$ of its largest entries through the greedy algorithm. As desired, the reconstructed shape is a smoothened version of the original one.}\label{fig:noise}
\end{figure}

We refer to \cite{Sigl:2013fk} for a detailed mathematical formulation of the model connecting the instantaneous star shape and its actual light intensity at different {  frequencies}. Here we limit ourselves to a schematic description where we assume the star being
a two dimensional object with a pulsating shape contour. Let the function $u(\varphi)$ {  describe} the star shape contour, for a parameter $-1\leq \varphi\leq 1$, which also simultaneously represents the temperature (or emitted wavelength) on the stellar surface at some fixed point in time. Its oscillatory behavior yields 
\begin{equation*}
u(\varphi) =\sum_{i=1}^d x_i \sin((2\pi\varphi+\theta)i),
\end{equation*}
for some coefficient vector $x=(x_1, \ldots,x_d)$ and some inclination angle $\theta$. This vector $x$ needs to be reconstructed from the instantaneous light measurements $b=(b_1,\ldots, b_n)$, modeled in \cite{Sigl:2013fk} by the formula 
\begin{equation}\label{eq:model}
{  b_l = \frac{\sqrt{\pi}}{2d+1}\sum_{j=-d}^d \omega_l\big (f_j  \sum_{k=1}^d x_k \sin((2\pi \frac{j}{d}+\theta)k) \big) f_j   \sum_{i=1}^d x_i \sin((2\pi \frac{j}{d}+\theta)i) ,\quad l=1,\ldots,n,}
\end{equation}
so that we suppose that $u$ is sampled at $j/d$. 
Here, $f$ is a correction factor modeling {\it limb darkening}, i.e., the fading intensity of the light of the star towards its limb,  and $\omega_l(\cdot)$ is some partition of unity modeling the wavelength range of each telescope sensor, see \cite{Sigl:2013fk} for details. Notice that the light intensity data \eqref{eq:model} at different frequency bands (corresponding to different  $\omega_l$) are obtained through the quasi-linear measurements $b = A(x) = F(x)x$ with
\begin{equation*}
F(x)_{l,i}:= \frac{\sqrt{\pi}}{2d+1}\sum_{j=-d}^d \omega_l\big (f_j  \sum_{k=1}^d x_k \sin((2\pi \frac{j}{d}+\theta)k) \big) f_j  \sin((2\pi \frac{j}{d}+\theta)i),
\end{equation*}
and one wants to reconstruct a vector $x$ matching the data with few nonzero entries. In fact it is rather accepted in the Asteroseismology community that only few low frequencies  of the star shape (when its contour shape is expanded in spherical harmonics) are relevant \cite{Aerts:2010fk}.

We do not claim that the model \eqref{eq:model} matches all of the assumptions in Theorem \ref{th:rec result 1}, but we shall observe that Algorithm \ref{algo:1} for $p=2$ can be used to identify the instantaneous pulsation patterns of simulated light intensity data, when low frequencies are activated. This is a consequence of the fact that different low frequency activations result in sufficiently uncorrelated measurements in the data to be distinguished by the greedy algorithm towards recovery. For the numerical experiments, the ambient dimension is $d=800$ and we make $n=13$ measurements, see \cite{Sigl:2013fk} for details on the choice of $\omega_l$, $f$, and the used multivariate optimization routines. We generate $2$- and $3$-sparse signals and apply Algorithm \ref{algo:1} in Figs.~\ref{fig:pulsation} and \ref{fig:pulsation 2} to reconstruct the signal. The greedy strategy identifies one additional location of the solution's support at each iteration step and finds the correct signal after $2$ and $3$ steps, respectively. In Fig.~\ref{fig:noise}, we generated a signal whose entries decay rapidly, so that higher frequencies have lower magnitudes. {  We show the reconstruction of the shape after $3$ iterations of  the greedy algorithm}. As expected, higher frequencies are suppressed and we obtain a {  low-pass filter} approximation of the original shape.

%

\subsection{Deterministic conditions II}

Experiments in X-ray crystallography and diffraction imaging require signal reconstruction from magnitude measurements, usually in terms of light intensities. We do not present the explicit physical models, which would go beyond the scope of the present paper, but refer to the literature instead. It seems impossible to provide a comprehensive list of references, so we only mention \cite{Drenth:2010fk,Fienup:1982vn,Gerchberg:1972kx} for some classical algorithms. 

Let $\hat{x}\in\R^d$ be some signal that we need to reconstruct from measurements $b=\big(|\langle a_i,\hat{x}\rangle |^2\big)_{i=1}^n$, where we selected a set of measurement vectors $\{a_i: i=1,\dots n\} \subset\R^d$. In other words, we have phaseless measurements and need to reconstruct $\pm \hat{x}$. It turns out {  surprisingly} that the above framework of {  reconstruction from nonlinear sensing} can be modified to fit the phase retrieval problem. 
{  Let us stress that so far the most efficient and stable recovery procedures are  based on semi-definite programming, as used in \cite{Candes:uq}, by recasting the problem into a perhaps demanding optimization on matrices. 
In this section we show that there is no need to linearize the problem by lifting the dimensionality towards low-rank matrix recovery, but is is sufficient to address a plain sparse vector recovery in the fully nonlinear setting.}

In models relevant to optical measurements like diffraction imaging and X-ray crystallography, we must deal with the complex setting, in which $x\in\C^d$ is at most determined up to multiplication by a complex unit. We shall state our findings for the real case first and afterwards discuss the extensions to complex vector spaces.

\subsubsection{Quasi-linear compressed sensing in phase retrieval}\label{sec:phase}

Let $\{A_i: i=1, \dots,n\}\subset\R^{d\times d}$ be a collection of measurement matrices. We consider the map
\begin{equation}\label{eq: def A Ai}
A: \R^d\rightarrow \R^{n}, \quad x\mapsto A(x)= F(x) x, \mbox{ where }
F(x) = \begin{pmatrix}  x^* A_1\\
\vdots \\
x^* A_n
\end{pmatrix},
\end{equation}
and we aim at reconstructing a signal vector $\hat{x}\in\R^d$ from measurements $b=A(\hat{x})$. 
Since $A(x)=A(-x)$, the vector $\hat{x}$ can at best be determined up to its sign, and the lower bound in \eqref{eq:1} cannot hold, not allowing us to use directly Theorem \ref{th:rec result 1}. {  However, we notice that for special classes of $A$, for instance, when $\{A_i: i=1, \dots,n\}$ are independent Gaussian matrices, the lower bound in \eqref{eq:1} holds with high probability as long as $y$ stays away from $-\hat{x}$, see the heuristic probability transitions of validity of \eqref{eq:1} shown in Fig.~\ref{fig:1}. Hence, there is the hope that the greedy algorithm can nevertheless be successful, because it proceeds by selecting first the largest components of the expected solution, hence orienting the reconstruction precisely towards the direction within the space where actually \eqref{eq:1} holds with high probability, in a certain sense realizing a self-fulfilling prophecy: the algorithm goes only where it is supposed to work.
In order to make this geometric intuition more explicit we shall modify the deterministic conditions of Theorem \ref{th:rec result 1} accordingly, so that we cover the above setting as well.}

\begin{figure}[h]
\centering
\subfigure[$k=2$]{
\includegraphics[width=.28\textwidth,trim=3.5cm 1cm 2.5cm .5cm,clip=true]{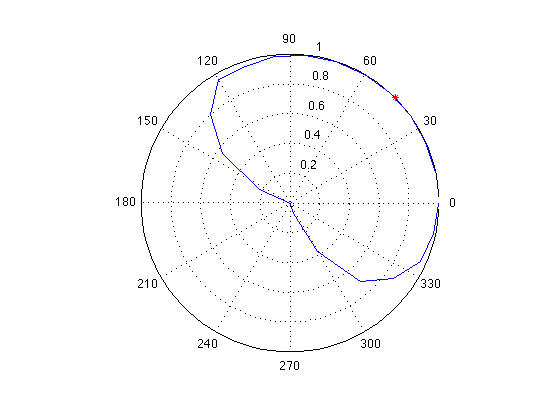}
\includegraphics[width=.28\textwidth,trim=3.5cm 1cm 2.5cm .5cm,clip=true]{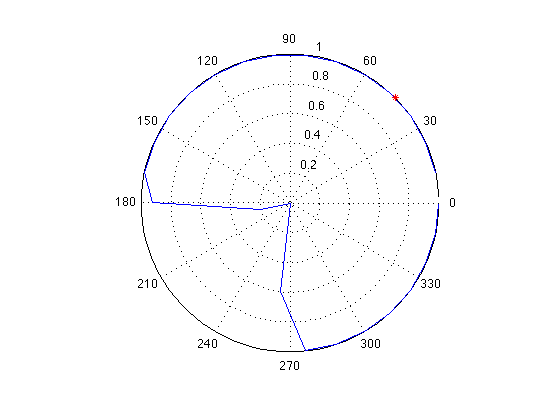}
\includegraphics[width=.28\textwidth,trim=3.5cm 1cm 2.5cm .5cm,clip=true]{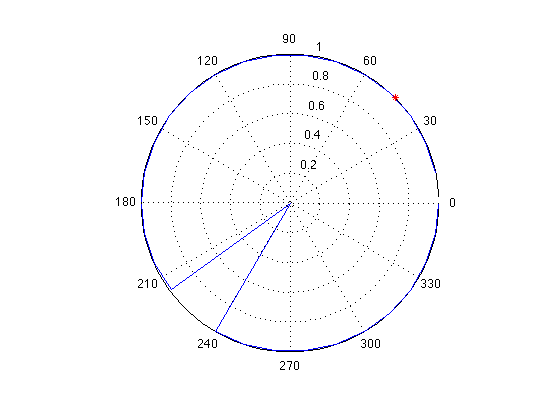}
}

\subfigure[$k=3$]{
\includegraphics[width=.32\textwidth,trim=1.3cm .5cm 1cm 0cm,clip=true]{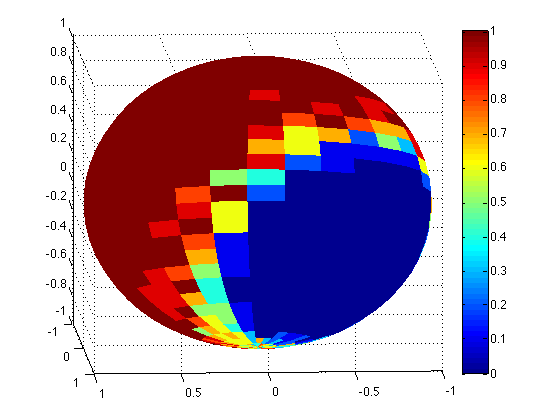}
\includegraphics[width=.32\textwidth,trim=1.3cm .5cm 1cm 0cm,clip=true]{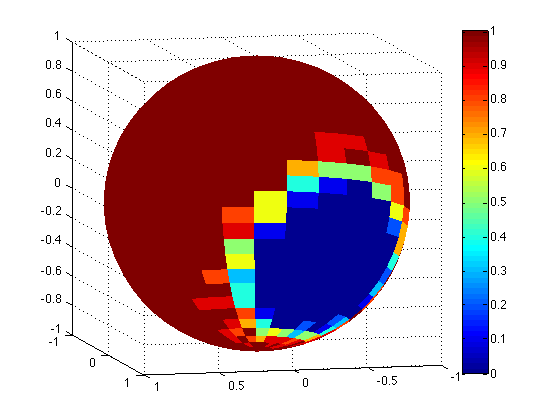}
\includegraphics[width=.32\textwidth,trim=1.3cm .5cm 1cm 0cm,clip=true]{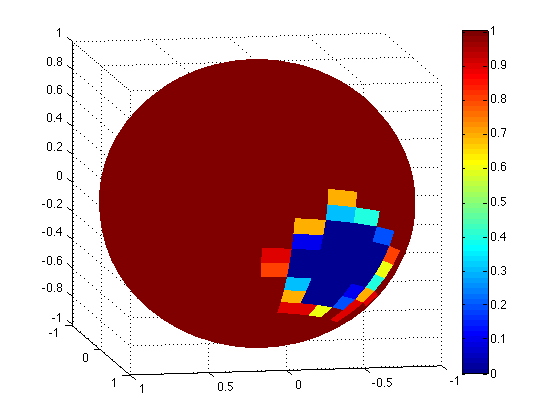}}
\caption{The map $A$ is chosen as in \eqref{eq: def A Ai} with  independent Gaussian matrices $\{A_i\}_{i=1}^n$. We plotted the success rates of the lower bound in \eqref{eq:1} for $k$-sparse $\hat{x}\in S^{d-1}$ marked in red and $y$ running through $S^{d-1}$ with the same sparsity pattern, so that both vectors can be visualized on a $k-1$-dimensional sphere (here for $k=1,2$). Parameters are $d=80$, $n=30$, and $\alpha$ decreases from left to right. 
}\label{fig:1}
\end{figure}

%
{  Under slightly different deterministic conditions}, we derive a recovery result very similar to Theorem \ref{th:rec result 1}:
\begin{theorem}\label{th:new}
Let $A$ be given by \eqref{eq: def A Ai} and $b=A(\hat{x})+e$, where $\hat{x}\in\R^d$ is the signal to be recovered and $e\in\R^n$ is a noise term. Suppose further that $1\leq k\leq d$, $r_k(\hat{x})\neq 0$, and $1\leq p<\infty$. If the following conditions are satisfied, \begin{itemize}
\item[(i) ]  there are constants $\alpha_k,\beta_k>0$, such that, for all $k$-sparse $y\in\R^d$,
\begin{equation}\label{eq:1b}
\alpha_k \|\hat{x}_{\{k\}}\hat{x}_{\{k\}}^*-yy^*\|_{HS}\leq \| A(\hat{x}_{\{k\}}) -A(y) \|_{\ell_p}
\leq \beta_k \|\hat{x}_{\{k\}}\hat{x}_{\{k\}}^*-yy^*\|_{HS},
\end{equation} 
\item[(ii) ] $\hat{x}\in\mathcal{D}_\kappa$ with $\kappa<\frac{\tilde{\alpha}_k}{\sqrt{\tilde{\alpha}_k^2+2(\beta_k+2L_k)^2}}$, where $0<\tilde{\alpha}_k\leq \alpha_k-2\|e\|_{\ell_p}/r_k(\hat{x})$ and $L_k\geq 0$ with  $\|A(\hat{x})-A(\hat{x}_{\{k\}})\|_{\ell_p}\leq L_k \|\hat{x}\hat{x}^*-\hat{x}_{\{k\}}\hat{x}_{\{k\}}^*\|_{HS}$,
\end{itemize}
then the $\ell_p$-greedy Algorithm \ref{algo:1} yields a sequence $(x^{(j)})_{j=1}^k$ satisfying $\supp(x^{(j)}) = \supp(\hat{x}_{\{j\}})$ and 
\begin{equation*}
\| x^{(j)}{x^{(j)}}^* -\hat{x}\hat{x}^*\|_{HS} \leq \|e\|_{\ell_p}/\alpha_k+ \kappa^jr_1(\hat{x}) \sqrt{3} (1+\frac{\beta_k+2L_k}{\alpha_k}).
\end{equation*}
If $\hat{x}$ is $k$-sparse, then $\|x^{(k)}-\hat{x}\|\leq \|e\|_{\ell_p}/\alpha_k$.
\end{theorem}

\begin{remark}
Note that \eqref{eq:1b} resembles the restricted isometry property for rank minimization problems, in which $\alpha_k$ and $\beta_k$ are required to be close to each other, see \cite{Goldfarb:2010fk,Recht:2010fk} and references therein. In Theorem \ref{th:new}, we can allow for any pair of constants and compensate deviation between $\alpha_k$ and $\beta_k$ by adding the decay condition on the signal. In other words, we shift conditions on the measurements towards conditions on the signal. 
\end{remark}

The structure of the proof of Theorem \ref{th:new} is almost the same as the one for Theorem \ref{th:rec result 1}, so that we first derive results similar to the Lemmas \ref{lemma:-1} and \ref{lemma:0}:
\begin{lemma}\label{lemma:-1b}
If $\hat{x}\in\R^d$ is contained in $\mathcal{D}_\kappa$, then 
\begin{equation}\label{eq:2ba0b}
\|\hat{x}\hat{x}^*-\hat{x}_{\{j\}}\hat{x}_{\{j\}}^*\|_{HS} < \sqrt{2}\|\hat{x}\|r_{j+1}(\hat{x}) \frac{1}{\sqrt{1-\kappa^2}} \leq \sqrt{2}\|\hat{x}\|r_{j}(\hat{x}) \frac{\kappa}{\sqrt{1-\kappa^2}}.
\end{equation}
\end{lemma}

We omit the straight-forward proof and state the second lemma that is needed: 
\begin{lemma}\label{lemma:1}
Fix $1\leq k\leq d$ and suppose that $r_k(\hat{x})\neq 0$. If $\hat{x}\in\R^d$ is contained in $\mathcal{D}_\kappa$ with $\kappa<\frac{\tilde{\alpha}}{\sqrt{\tilde{\alpha}^2+2(\beta+2L)^2}}$, where $0<\tilde{\alpha}\leq \alpha-2\|e\|_{\ell_p}/r_k(\hat{x})$, for some $\alpha,\beta,L>0$, then, for $j=1,\ldots,k$,
\begin{equation}\label{eq:2b}
\alpha \|\hat{x}\| r_j(\hat{x})>2\|e\|_{\ell_p}+\beta\|\hat{x}_{\{j\}}\hat{x}^*_{\{j\}}-\hat{x}_{\{k\}}\hat{x}_{\{k\}}^*\|_{HS}+2L\|\hat{x}_{\{k\}}\hat{x}_{\{k\}}^*-\hat{x}\hat{x}^*\|_{HS}
\end{equation}
\end{lemma}

\begin{proof}[Proof of Lemma \ref{lemma:1}]
As in the proof of Lemma \ref{lemma:0}, the conditions on $\kappa$ imply
\begin{equation*}
\alpha-\frac{2\|e\|_{\ell_p}}{\|\hat{x}\| r_k(\hat{x})} - \frac{\kappa}{\sqrt{1-\kappa^2}} \sqrt{2}(\beta+2L)>0.
\end{equation*}
Multiplying by $\|\hat{x}\| r_j(\hat{x})$ and applying Lemma \ref{lemma:-1b} yield
\begin{equation*}
\alpha\|\hat{x}\| r_j(\hat{x})-2\|e\|_{\ell_p} - (\beta+2L)\|\hat{x}_{\{j\}}\hat{x}^*_{\{j\}}-\hat{x}\hat{x}^*\|_{HS}>0.
\end{equation*}
%
%
%
%
%
%
 We can further estimate
 \begin{equation*}
\alpha\|\hat{x}\| r_j(\hat{x})-2\|e\|_{\ell_p}  -\beta\|\hat{x}_{\{j\}}\hat{x}^*_{\{j\}}-\hat{x}_{\{k\}}\hat{x}_{\{k\}}^*\|_{HS}-2L\|\hat{x}_{\{k\}}\hat{x}_{\{k\}}^*-\hat{x}\hat{x}^*\|_{HS}>0,
 \end{equation*}
which concludes the proof.
\end{proof}


\begin{proof}[Proof of Theorem \ref{th:new}]
As in the proof of Theorem \ref{th:rec result 1}, we must check that the index set selected by the $\ell_p$-greedy Algorithm \ref{algo:1} matches the location of the nonzero entries of $\hat{x}_{\{k\}}$. Again, we use induction and the initialization $j=0$ is trivial. Now, we suppose that $\Lambda^{(j-1)}\subset \supp(\hat{x}_{\{j-1\}})$ and choose $l\not\in \supp(\hat{x}_{\{j\}})$. The lower bound in \eqref{eq:1b} yields
\begin{align*}
\|A(x^{(j,l)})-A(\hat{x})\|_{\ell_p} & \geq \|A(x^{(j,l)})-A(\hat{x}_{\{k\}})\|_{\ell_p} -\|A(\hat{x}_{\{k\}})-A(\hat{x})\|_{\ell_p}  \\
& \geq   \alpha_k\| x^{(j,l)}{x^{(j,l)}}^*-\hat{x}_{\{k\}}\hat{x}_{\{k\}}^* \|_{HS}	-L_k\|\hat{x}_{\{k\}}\hat{x}_{\{k\}}^*-\hat{x}\hat{x}^*\|_{HS}	\\
& \geq \alpha \|\hat{x}_{\{k\}}\|  r_j(\hat{x})-L_k\|\hat{x}_{\{k\}}\hat{x}_{\{k\}}^*-\hat{x}\hat{x}^*\|_{HS},
\end{align*}
which is due to $l\not\in \supp(\hat{x}_{\{j\}})$, so that one row and one column of $x^{(j,l)}{x^{(j,l)}}^*$ corresponding to one of the $j$-largest entries of $\hat{x}$ are zero. Lemma \ref{lemma:1} implies
\begin{equation}\label{eq:op new 1}
\|A(x^{(j,l)})-A(\hat{x})\|_{\ell_p} > 2\|e\|_{\ell_p}+\beta_k\|\hat{x}_{\{j\}}\hat{x}^*_{\{j\}}-\hat{x}_{\{k\}}\hat{x}_{\{k\}}^*\|_{HS}+L_k\|\hat{x}_{\{k\}}\hat{x}_{\{k\}}^*-\hat{x}\hat{x}^*\|_{HS}.
\end{equation}
On the other hand, the minimizing property of $x^{(j)}$ and the Condition \eqref{eq:1b} imply
\begin{align*}
\|A(x^{(j)})-A(\hat{x})\|_{\ell_p}& \leq 	\|e\|_{\ell_p}+\|A(x^{(j)})-b\|_{\ell_p}	\\
&\leq \|e\|_{\ell_p}+\|A(\hat{x}_{\{j\}})-b\|_{\ell_p}	\\
&\leq 2\|e\|_{\ell_p}+\|A(\hat{x}_{\{j\}})-A(\hat{x}_{\{k\}})\|_{\ell_p}+	\|A(\hat{x}_{\{k\}})-A(\hat{x})\|_{\ell_p}\\
&\leq 2\|e\|_{\ell_p}+\beta_k\|\hat{x}_{\{j\}}\hat{x}_{\{j\}}^*-\hat{x}_{\{k\}}\hat{x}_{\{k\}}^*\|_{HS}+	L_k\|\hat{x}_{\{k\}}\hat{x}_{\{k\}}^*-\hat{x}\hat{x}^*\|_{HS}.
\end{align*}
The latter inequality implies with \eqref{eq:op new 1} that $x^{(j)} = x^{(j,l)}$, for all $l\in \supp(\hat{x}_{\{j\}})$, which concludes the part about the support. 

Next, we shall verify the error bound. We obtain
\begin{align*}
\| x^{(j)}{x^{(j)}}^* -\hat{x}\hat{x}^*\|_{HS}  & \leq \| x^{(j)}{x^{(j)}}^* -\hat{x}_{\{k\}}\hat{x}_{\{k\}}^*\|_{HS}  + \| \hat{x}_{\{k\}}\hat{x}_{\{k\}}^*-\hat{x}\hat{x}^*\|_{HS}\\
&\leq 1/\alpha_k \|  A(x^{(j)}) - A(\hat{x}_{\{k\}})\|_{\ell_p} + \| \hat{x}_{\{k\}}\hat{x}_{\{k\}}^*-\hat{x}\hat{x}^*\|_{HS}\\
& \leq  1/\alpha_k  \|A(x^{(j)})-A(\hat{x})\|_{\ell_p}+1/\alpha_k\|A(\hat{x})-A(\hat{x}_{\{k\}})\|_{\ell_p}\\
& \qquad +\| \hat{x}_{\{k\}}\hat{x}_{\{k\}}^*-\hat{x}\hat{x}^*\|_{HS}\\
& \leq 1/\alpha_k  \|A(\hat{x}_{\{j\}})-A(\hat{x})\|_{\ell_p}+\|e\|_{\ell_p}/\alpha_k+(L_k/\alpha_k+1)\|\hat{x}\hat{x}^*-\hat{x}_{\{k\}}\hat{x}_{\{k\}}^*\|_{HS}\\
&\leq 1/\alpha_k  \|A(\hat{x}_{\{j\}})-A(\hat{x}_{\{k\}})\|_{\ell_p}+1/\alpha_k  \|A(\hat{x}_{\{k\}})-A(\hat{x})\|_{\ell_p}+\|e\|_{\ell_p}/\alpha_k\\
& \qquad+(L_k/\alpha_k+1)\|\hat{x}\hat{x}^*-\hat{x}_{\{k\}}\hat{x}_{\{k\}}^*\|_{HS}\\
&\leq \beta_k/\alpha_k  \|\hat{x}_{\{j\}}\hat{x}_{\{j\}}^*-\hat{x}_{\{k\}}\hat{x}_{\{k\}}^*\|_{HS}+\|e\|_{\ell_p}/\alpha_k\\
&\qquad +(2L_k/\alpha_k+1)\|\hat{x}\hat{x}^*-\hat{x}_{\{k\}}\hat{x}_{\{k\}}^*\|_{HS}\\
& \leq \big(\beta_k/\alpha_k+2L_k/\alpha_k+1  \big)\|\hat{x}\hat{x}^*-\hat{x}_{\{j\}}\hat{x}_{\{j\}}^*\|_{HS}+\|e\|_{\ell_p}/\alpha_k\\
& \leq \big(\beta_k/\alpha_k+2L_k/\alpha_k+1  \big)\sqrt{2}r_{j+1}(\hat{x}) \frac{1}{\sqrt{1-\kappa^2}}+\|e\|_{\ell_p}/\alpha_k\\
& \leq \big(\beta_k/\alpha_k+2L_k/\alpha_k+1  \big) \kappa^j\sqrt{2}r_1(\hat{x}) \frac{1}{\sqrt{1-\kappa^2}}+\|e\|_{\ell_p}/\alpha_k.
\end{align*}
Some rough estimates yield
\begin{equation*}
\frac{\beta_k/\alpha_k+2L_k/\alpha_k+1  }{\sqrt{1-\kappa^2}}\sqrt{2}\leq \sqrt{3} (1+\frac{\beta_k+2L_k}{\alpha_k}),
\end{equation*}
which concludes the proof.
\end{proof}

\begin{remark}
The greedy Algorithm \ref{algo:1} can also be performed in the complex setting. 
The complex version of Theorem \ref{th:new} also holds when recovery of $\pm \hat{x}$ is replaced {  by a complex unit vector} times $\hat{x}$. 
\end{remark}

\subsection{Signal recovery from random measurements}
{  We aim at choosing $\{A_i: i=1, \dots,n\}$ in a suitable random way, so  that the} conditions in Theorem \ref{th:new} are satisfied with high probability. Indeed, the upper bound in \eqref{eq:1b} is always satisfied for some $C>0$, because we are in a finite-dimensional regime. In this section we are interested in one rank matrices $A_i=a_ia_i^*$, for some vector $a_i\in\R^d$, $i=1,\ldots,n$, because then
\begin{equation}\label{eq:model phase}
A(\hat{x})=\big(\hat{x}^*A_1\hat{x},\ldots,\hat{x}^*A_n\hat{x}\big)^* = (|\langle a_1,\hat{x}\rangle|^2,\ldots,|\langle a_n,\hat{x}\rangle|^2)^\top 
\end{equation}
models the phase retrieval problem. 

\subsubsection{Real random measurement vectors}
To check on the assumptions in Theorem \ref{th:new}, we shall {  draw at random} the measurement vectors $\{a_i: i=1,\dots,n\}$  from {  probability distributions} to be characterized next. We say that a random vector $a\in\R^d$ satisfies the \emph{small-ball assumption} if there is a constant $c>0$, such that, for all $z\in\R^d$ and $\varepsilon>0$,
\begin{equation*}
\mathbb{P}\big ( |\langle a,z\rangle |\leq \varepsilon \|z\|\big)  \leq c \varepsilon.
\end{equation*}
Moreover, we say that $a$ is \emph{isotropic} if $\mathbb{E}|\langle a,z\rangle|^2 = \|z\|^2$, for all $z\in\R^d$. The vector $a$ is said to be \emph{$L$-subgaussian} if, for all $z\in\R^d$ and $t\geq 1$, 
\begin{equation*}
\mathbb{P}\big(  |\langle a,z\rangle| \geq  t L\|z\| \big) \leq 2e^{-t^2/2}.
\end{equation*}
Eldar and Mendelson derived the following result:
\begin{theorem}[\cite{Eldar:2012fk}]\label{th:eldar}
Let $\{a_i: i=1,\dots,n\}$ be a set of independent copies of a random vector $a\in\R^d$ that is isotropic, $L$-subgaussian, and satisfies the small-ball assumption. Then there are positive constants $c_1,\ldots,c_4$ such that, for all $t\geq c_1$ with $n\geq  k c_2 t^3 \log(ed/k)$, and for all $k$-sparse $x,y\in\R^d$, 
\begin{equation}\label{eq:eldar etc}
\sum_{i=1}^n \big| |\langle a_i,x\rangle|^2 - |\langle a_i,y\rangle|^2 \big| \geq c_4 \|x-y\| \|x+y\|
\end{equation}
with probability of failure at most $2e^{-k c_3t^2 \log(ed/k)}$.
\end{theorem}

The uniform distribution on the sphere and the Gaussian distribution on $\R^d$ induce random vectors satisfying the assumptions of Theorem \ref{th:eldar}. However, at first glance, the above theorem does not help us directly, because we are seeking for an estimate involving the Hilbert-Schmidt norm. It is remarkable though that 
\begin{equation}\label{eq:HS with pm}
\|xx^*-yy^*\|_{HS} \leq  \|x-y\| \|x+y\|\leq \sqrt{2}\|xx^*-yy^*\|_{HS}.
\end{equation}
The relation \eqref{eq:HS with pm} then yields that also the lower bound in \eqref{eq:1b} is satisfied for any constant $\alpha\leq c_4$, so that Theorem \ref{th:new} can be used for $p=1$. Thus, if the signal $\hat{x}$ is sparse and satisfies decay conditions matching the constants in \eqref{eq:1b}, then Algorithm \ref{algo:1} for $p=1$ recovers $x^{(k)}=\pm \hat{x}$.

\subsubsection{Complex random measurement vectors} 
In the following and at least for the uniform distribution on the sphere, we shall generalize Theorem \ref{th:eldar} to the complex setting, so that the assumptions of Theorem \ref{th:new} hold with high probability. 
\begin{theorem}\label{th:complex and new}
If $\{a_i: i=1,\dots,n\}$ are independent uniformly distributed vectors on the unit sphere, then there is a constant $\alpha>0$ such that, for all $k$-sparse $x,y\in\C^d$ and $n\geq c_1k\log(ed/k)$,
\begin{equation*}
\sum_{i=1}^n \big| |\langle a_i,x\rangle|^2 - |\langle a_i,y\rangle|^2 \big| \geq \alpha n \|xx^*-yy^*\|_{HS}
\end{equation*}
with probability of failure at most $e^{-nc_2}$.
\end{theorem}

\begin{proof}
For fixed $x,y\in\R^d$, the results in \cite{Candes:uq} imply that there are constants $c_1, c>0$ such that, for all $t>0$, 
\begin{equation}\label{eq:not uniform yet}
\sum_{i=1}^n\big| |\langle a_i,x\rangle|^2 - |\langle a_i,y\rangle|^2\big|\geq 1/\sqrt(2)(c_1-t)n\|xx^*-yy^*\|_{HS}
\end{equation}
with probability of failure at most $2e^{-n c t^2}$. If both $x$ and $y$ are $k$-sparse, then the union of their supports induces an at most $2k$-dimensional coordinate subspace, so that also $xx^*-yy^*$ can be reduced to a $2k\times 2k$ matrix, by eliminating rows and columns that do not belong to the indices of the subspace. Results in \cite{Candes:uq} can be used to derive that the estimate \eqref{eq:not uniform yet} holds uniformly for elements in this subspace when $n\geq c_3 t^{-2}\log(t^{-1})2k$, for some constant $c_3>0$. There are at most $\binom{d}{2k}\leq (ed/(2k))^{2k}$ many of such coordinate subspaces, see also \cite{Baraniuk:2008fk}. Therefore, {  by a union bound}, the probability of failure is at most
\begin{equation*}
(ed/(2k))^{2k}2e^{-n c t^2} =  2e^{-n\big(ct^2-\frac{2k\log(ed/(2k))}{n}\big)}.
\end{equation*}
Thus, if also $n\geq \frac{1}{ct^2-c_2}\log(ed/2k)2k$ with $ct^2-c_2>0$, then we have the desired result.
\end{proof}

\begin{remark}
We want to point out that the use of the term $\|x-y\|\|x+y\|$ limits Theorem \ref{th:eldar} to the real setting. Our observation \eqref{eq:HS with pm} was the key to derive the analog result in the complex setting. 
 \end{remark}

\subsubsection{Rank-$m$ projectors as measurements} 
 A slightly more general phase retrieval problem was discussed in \cite{Bachoc:2012fk}, where the measurement $A$ in \eqref{eq: def A Ai}  is given by $A_i=\frac{d}{m}P_{V_i}$, $i=1,\ldots,n$, and each $P_{V_i}$ is an orthogonal projector onto an $m$-dimensional linear subspace $V_i$ of $\R^d$. The set of all $m$-dimensional linear subspaces $\mathcal{G}_{m,d}$ is a manifold endowed with the standard normalized Haar measure $\sigma_m$:
 \begin{theorem}[\cite{Bachoc:2012fk}]\label{th:not sparse bachoc}
 There is a constant $u_m>0$, only depending on $m$, such that the following holds: for $0<r<1$ fixed, there exist constants $c(r),C(r)>0$, such that, for all $n\geq c(r)d$ and $\{V_j: j=1,\dots n\} \subset\mathcal{G}_{m,d}$ independently chosen random subspaces with identical distribution $\sigma_m$,  the inequality 
 \begin{equation}\label{eq:mine in 3}
\sum_{i=1}^n \big| x^*A_ix-y^*A_iy \big|
\geq u_m (1-r)n\|xx^*-yy^*\|_{2},
\end{equation}
for all $x,y\in\R^d$, holds with probability of failure at most $e^{-C(r)n} $.  
 \end{theorem}
 
Since the rank of $xx^*-yy^*$ is at most $2$, its Hilbert-Schmidt norm is bounded by $\sqrt{2}$ times the operator norm. Thus, we have the lower $\ell_1$-bound in \eqref{eq:1b} for $k=d$ when $n\geq c(r)d$, and the random choice of subspaces (and hence orthogonal projectors) enables us to apply Theorem \ref{th:new}. The result for $k$-sparse signals is a consequence of Theorem \ref{th:not sparse bachoc}:
\begin{corollary}\label{cor:new}
There is a constant $u_m>0$, only depending on $m$, such that the following holds: for $0<r<1$ fixed, there exist constants $c(r),C(r)>0$, such that, for all $n\geq c(r)k\log(ed/k)$ and $\{V_j: j=1,\dots n\}\subset\mathcal{G}_{m,d}$ independently chosen random subspaces with identical distribution $\sigma_m$,  the inequality 
 \begin{equation}\label{eq:in theorem new again}
\sum_{i=1}^n \big| x^*A_ix-y^*A_iy \big|
\geq u_m (1-r)n\|xx^*-yy^*\|_{HS},
\end{equation}
for all $k$-sparse $x,y\in\R^d$, holds with probability of failure at most $e^{-C(r)n}$.  
\end{corollary}

\begin{proof}
The lower bound on $n$ in Theorem \ref{th:not sparse bachoc} is not needed when the vectors $x$ and $y$ are fixed in \eqref{eq:mine in 3}, cf.~\cite{Bachoc:2012fk}. If both $x,y$ are supported in one fixed coordinate subspace of dimension $2k$, then the proof of Theorem \ref{th:not sparse bachoc} in \cite{Bachoc:2012fk}, see also \cite{Candes:uq}, yields that \eqref{eq:in theorem new again} holds uniformly for this subspace provided $n\geq c(r)2k$.

Similar to the proof of Theorem \ref{th:complex and new}, we shall use Theorem \ref{th:not sparse bachoc} with a $2k$ coordinate subspace and then apply a union bound by counting the number of such subspaces. Indeed, since $xx^*-yy^*$ can be treated as a $2k\times 2k$ matrix by removing zero rows and columns, Theorem \ref{th:not sparse bachoc} implies \eqref{eq:in theorem new again} for all $x,y\in\R^d$ supported in a fixed coordinate subspace of dimension $2k$ with probability of failure at most $e^{-C(r)n}$ when $n\geq c(r)2k$. Again, we have used that $xx^*-yy^*$ has rank at most two, so that the Hilbert-Schmidt norm is bounded by $\sqrt{2}$ times the operator norm. The remaining part can be copied from the end of the proof of Theorem \ref{th:complex and new}.
\end{proof}

\begin{remark}
It is mentioned in \cite{Bachoc:2012fk} already that Theorem \ref{th:not sparse bachoc} also holds in the complex setting. Therefore, Corollary \ref{cor:new} has a complex version too, and our present results hold for complex rank-$m$ projectors. 
\end{remark}

\subsubsection{Nearly isometric random maps}
To conclude the discussion on random measurements for phase retrieval, we shall generalize some results from \cite{Recht:2010fk} to sparse vectors. Also, we want to present a framework, in which $\{A_i: i=1, \dots,n\}$ in \eqref{eq:model phase} can be chosen as a set of independent random matrices with independent Gaussian entries. Let $\mathcal{A}$ be a random map that takes values in linear maps from $\R^{d\times d}$ to $\R^n$. Then $\mathcal{A}$ is called nearly isometrically distributed if, for all $X\in\R^{d\times d}$, 
\begin{equation}\label{eq:cond 1 again}
\mathbb{E} \|\mathcal{A}(X)\|^2 = \|X\|_{HS}^2,
\end{equation}
and, for all $0<\epsilon<1$, we have
\begin{equation}\label{eq:def iso}
(1-\epsilon) \|X\|_{HS}^2\leq  \|\mathcal{A}(X)\|^2\leq (1+\epsilon)  \|X\|_{HS}^2
\end{equation}
with probability of failure at most $2 e^{-n f(\epsilon)}$,  where $f:(0,1)\rightarrow \R_{>0}$ is an increasing function.
%

Note that the definition of nearly isometries  in \cite{Recht:2010fk} is more restrictive, but we can find several examples there. For instance, if $\{A_i: i=1, \dots,n\}$ in \eqref{eq:model phase} are independent matrices with independent standard Gaussian entries, then the map
\begin{equation*}
\mathcal{A}:\R^{d\times d}\rightarrow \R^n,\qquad \mathcal{A}(X):=\frac{1}{\sqrt{n}}\begin{pmatrix}
\trace(A_1^*X) \\
\vdots\\
\trace(A_n^*X)
\end{pmatrix}
\end{equation*}
is nearly isometrically distributed, see \cite{Dasgupta:2003fk,Recht:2010fk}. 

The following theorem fits into our setting, and it should be mentioned that we will only use \eqref{eq:cond 1 again} and \eqref{eq:def iso} for symmetric matrices $X$ of rank at most $2$. So, we could even further weaken the notion of nearly isometric distributions accordingly.
\begin{theorem}\label{th:again}
Fix $1\leq k\leq d$. If $\mathcal{A}$ is a nearly isometric random map from $\R^{d\times d}$ to $\R^n$ and $A(x):=\mathcal{A}(xx^*)$, then there are constants $c_1,c_2>0$, such that, uniformly for all $0<\delta<1$ and all $k$-sparse $x,y\in\R^d$, 
\begin{equation}\label{eq:final new I hope}
(1-\delta)\|xx^*-yy^*\|_{HS} \leq \|A(x) -A(y)\| \leq (1+\delta)\|xx^*-yy^*\|_{HS} 
\end{equation}
with probability of failure at most $2e^{-n\big(f(\delta/2)  -\frac{4k+2}{n}\log(65/\delta) -\frac{2k}{n}\log(ed/2k)\big)}$.
\end{theorem}

As a consequence of Theorem \ref{th:again}, we can fix $\delta$ and derive two constants $c_1,c_2>0$ depending only on $\delta$ and $f(\delta/2)$, such that \eqref{eq:final new I hope} holds for all $k$-sparse $x,y\in\R^d$ in a uniform fashion with probability of failure at most $e^{-c_1n}$ whenever $n\geq c_2k\log(ed/k)$. 
The analogous result for not necessarily sparse vectors is derived in \cite{Recht:2010fk}. 
\begin{proof}
We fix an index set $\mathcal{I}$ of $2k$ coordinates in $\R^d$, denote the underlying coordinate subspace by $V$, and define 
\begin{equation*}
\mathcal{X}=\{X\in\R^{d\times d} :  X=X^*,\;\rank(X)\leq 2,\;\|X\|_{HS} = 1,\; X_{i,j}=0, \text{ if $i\not\in 
\mathcal{I}$ or $j\not\in 
\mathcal{I}$}\}.
\end{equation*}
Any element $X\in\mathcal{X}$ can be written as $X=a xx^*+byy^*$ such that $a^2+b^2=1$ and $x,y\in V$ are orthogonal unit norm vectors. In order to build a covering of $\mathcal{X}$, we start with a covering of the $2k$-dimensional unit sphere $S_V$ in $ V$. Indeed, there is a finite set $\mathcal{N}_1\subset S_V$ of cardinality at most $(1+\frac{64}{\delta})^{2k}\leq (\frac{65}{\delta})^{2k}$, such that, for every $x\in S_V$, there is $y\in \mathcal{N}_1$ with $\|x-y\|\leq \delta/32$, see, for instance, \cite[Lemma 5.2]{Vershynin:2012fk}.
 We can also uniformly cover $[-1,1]$ with a finite set of cardinality $\frac{32}{\delta}$, so that the error is bounded by $\delta/16$. Thus, we can cover $\mathcal{X}$ with a set $\mathcal{N}$ of cardinality at most $(\frac{32}{\delta})^2 (\frac{65}{\delta})^{4k}\leq (\frac{65}{\delta})^{4k+2}$, such that, for every $X=a xx^*+byy^*\in\mathcal{X}$, there is $Y=a_0x_0x_0+b_0y_0y_0^*\in \mathcal{N}$ with 
\begin{equation}\label{eq:covering}
|a-a_0|\leq \delta/16,\quad |b-b_0|\leq \delta/16, \quad \|x-x_0\|\leq \delta/32,\quad \|y-y_0\|\leq \delta/32.
\end{equation}
%
We can choose $\varepsilon=\delta/2$, so that \eqref{eq:def iso} holds uniformly on $\mathcal{N}$ with probability of failure at most $2 (\frac{65}{\delta})^{4k+2}e^{-nf(\delta/2)}$. Taking the square root yields with at most the same probability of failure that
\begin{equation*}
1-\delta/2\leq \|\mathcal{A}(Y)\| \leq 1+\delta/2
\end{equation*}
holds uniformly for all $Y\in\mathcal{N}$.

We now define the random variable 
\begin{equation}\label{eq:bound M}
M = \max\{ K\geq 0 : \|\mathcal{A}(X)\|\leq K \|X\|_{HS},\text{ for all } X\in\mathcal{X}\}
\end{equation}
and consider an arbitrary $X=axx+byy^*\in\mathcal{X}$. Then there is $Y=a_0x_0x_0+b_0y_0y_0^*\in\mathcal{N}$ such that \eqref{eq:covering} is satisfied, so that
we can further estimate
\begin{align*}
\|\mathcal{A}(X)\| & \leq \|\mathcal{A}(Y)\| + \|\mathcal{A}(axx+byy^*-a_0x_0x_0-b_0y_0y_0^*)\|\\
& \leq 1+\delta/2+   \|\mathcal{A}(axx-a_0x_0x_0^*)\|+  \|\mathcal{A}(byy-b_0y_0y_0^*)\|.
\intertext{Although $axx-a_0x_0x_0^*$ and $byy-b_0y_0y_0^*$ may not be elements in $\mathcal{X}$, a simple normalization allows us to apply the bound \eqref{eq:bound M}, so that we obtain}
\|\mathcal{A}(X)\| & \leq 1+\delta/2+   M\|axx-a_0x_0x_0^*\|_{HS}+  M\|byy-b_0y_0y_0^*\|_{HS} \\
& \leq 1+\delta/2+   M|a|\| xx^*- x_0x_0^*\|_{HS} +M |a-a_0|\| x_0x_0^*\|_{HS} +\\
& \qquad M |b|\| yy^*- y_0y_0^*\|_{HS} + M|b-b_0|\| y_0y_0^*\|_{HS} \\
&\leq  1+\delta/2 +M\delta/4.
\end{align*}
For the last inequality, we have used \eqref{eq:HS with pm}.   
By choosing $X\in\mathcal{X}$ with $\|\mathcal{A}(X)\|=M$, we derive $M\leq 1+\delta/2 +M\delta/4$, which implies $M\leq 1+\delta$. Thus, we obtain $\|\mathcal{A}(X)\| \leq  1+\delta$. 

The lower bound is similarly derived by 
\begin{equation*}
\|\mathcal{A}(X)\|\geq \|\mathcal{A}(Y)\| - \|\mathcal{A}(Y-X)\|\geq 1-\delta/2-(1+\delta)\delta/4\geq 1-\delta.
\end{equation*}

So far, we have the estimate \eqref{eq:final new I hope} in a uniform fashion for all $x,y$ in some fixed coordinate subspace of dimension $2k$ with probability of failure at most $2 (\frac{65}{\delta})^{4k+2}e^{-nf(\delta/2)}$. Again, we derive the union bound by counting subspaces as in the proof of Theorem \ref{th:complex and new}. There are at most $(ed/(2k))^{2k}$ many subspaces, so that we can conclude the proof. 
%
%
\end{proof}

 \subsection{Numerical experiments for greedy phase retrieval}
We shall follow the model in \eqref{eq:model phase} and study signal reconstruction rates for random choices of measurement vectors $\{a_i: i=1,\dots n\}$ chosen as independent standard Gaussian vectors. For a fixed number of measurements $n$, we shall study the signal recovery rate depending on the sparsity $k$. We expect to have high success rates for small $k$ and decreased rates when $k$ grows. Figure \ref{fig:phase} shows results of numerical experiments consistent with our theoretical findings. We must point out though that each step of the greedy algorithm requires solving a nonconvex global optimization problem. Here, we used standard optimization routines that may yield results that are not optimal. Better outcomes can be expected when applying more sophisticated global optimization methods, for instance, based on adaptive grids and more elaborate analysis of functions of few parameters in high dimensions, cf.~\cite{Cohen:2013uq}.

\begin{figure}
\centering
\subfigure[$d=20$, $n=7$]{
\includegraphics[width=.31\textwidth]{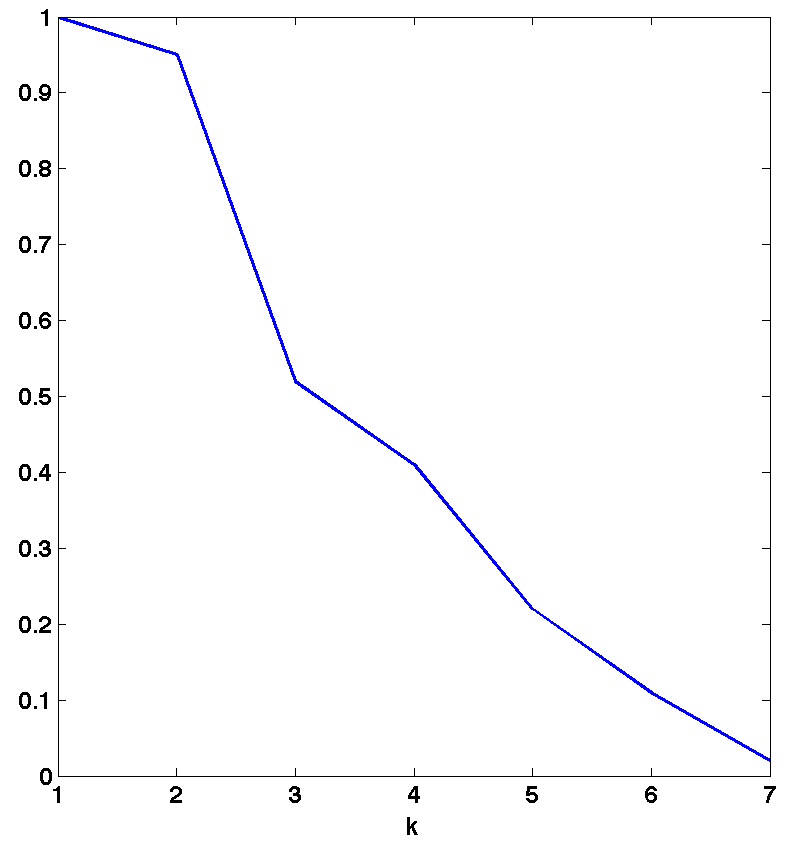}}
\subfigure[$d=20$, $n=11$]{
\includegraphics[width=.315\textwidth]{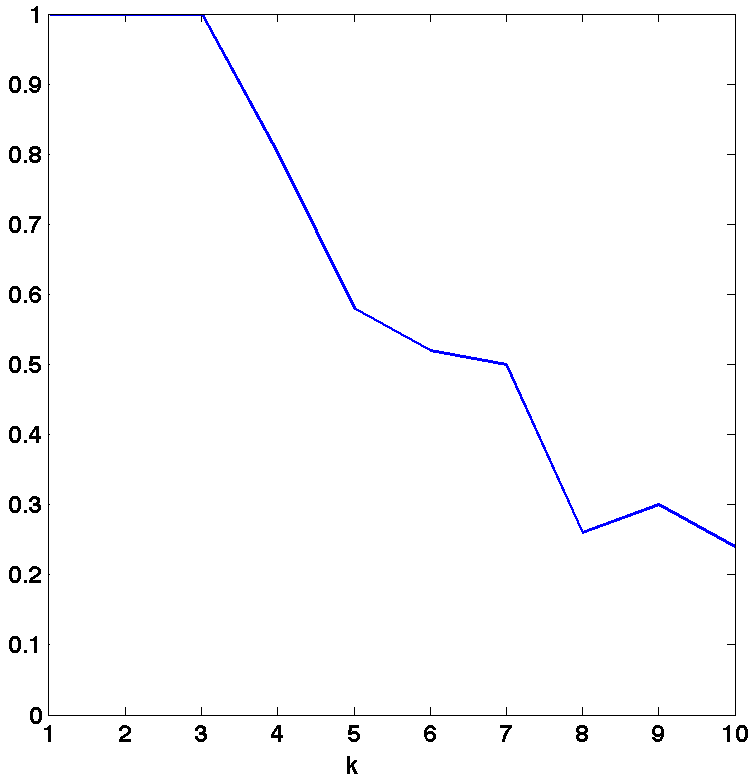}}
\caption{Signal reconstruction rates vs.~sparsity $k$. 
Reconstruction is repeated $50$ times for each signal and $k$ to derive stable recovery rates. For small $k$, we almost certainly recover the signal, and, as expected, increasing $k$ leads to decreased success rates.
}\label{fig:phase}
\end{figure}

 \section{Iterative thresholding for quasi-linear problems}\label{sec:thresholding}

{  Although the quasi-linear structure of the measurements does not play an explicit role in the formulation of Algorithm \ref{algo:1}, in the examples we showed in the previous sections it was nevertheless a crucial aspect to obtain generalized RIP conditions 
such as  \eqref{eq:1} and \eqref{eq:1b}. When using iterative thresholding algorithms, as we shall show below, the quasi-linear structure of the measurements gets into the formulation of the algorithms as well.
Hence, from now on we shall be a bit more explicit on the form of nonlinearity and we  consider a map $F:\R^d\rightarrow \R^{n\times d}$, and aim to reconstruct $\hat{x}\in\R^d$ from measurements $b\in\R^n$ given by 
\begin{equation*}
 F(\hat{x})\hat{x} = b.
\end{equation*}
As guiding examples, we keep as references the maps  $A$ in Proposition \ref{th:example}, which can be written as $A(x)=F(x)x$, where $F(x)=A_1+\epsilon f(\|x-x_0\|)A_2$. 
\\

The study of iterative thresholding algorithms that we propose below is motivated by the intrinsic limitations of Algorithm \ref{algo:1}, in particular, its restriction to recovering only signals which have a rapid decay of their nonincreasing rearrangement, and the
potential complexity explosion due to the need of performing several high-dimensional global optimizations at each greedy step.
}

\subsection{Iterative hard-thresholding}
We follow ideas in \cite{Blumensath:2012fk} and aim to use the iterative scheme 
\begin{equation}\label{eq:iteration}
x^{(j+1)} : = \big( x^{(j)}+\frac{1}{\mu_k} F(x^{(j)})^*(b- F(x^{(j)})x^{(j)}) \big)_{\{k\}},
\end{equation}
where $\mu_k>0$ is some parameter and, as before, $y_{\{k\}}$ denotes the best $k$-sparse approximation of $y\in\R^d$. The following theorem is a reformulation of a result by Blumensath in \cite{Blumensath:2012fk}:
\begin{theorem}\label{th:blumi}
 Let $b=A(\hat{x})+e$, where $\hat{x}\in\R^d$ is the signal to be recovered and $e\in\R^n$ is a noise term. Suppose $1\leq k\leq d$ is fixed. If $F$ satisfies the following assumptions,
 \begin{itemize}
 \item[(i) ] there is $c>0$ such that $\|F(\hat{x})\|_2\leq c$,
 \item[(ii) ] there are $\alpha_k,\beta_k>0$ such that, for all $k$-sparse $x,y,z\in\R^d$,
 \begin{equation}\label{eq:RIP again}
 \alpha_k \|x-y\|\leq \|F(z)(x-y)\|\leq \beta_k\|x-y\|,
 \end{equation}
 \item[(iii) ] there is $C_k>0$ such that, for all $k$-sparse $y\in\R^d$, 
 \begin{equation}\label{eq:new 1 new}
 \|F(\hat{x}_{\{k\}}) - F(y)\|_2 \leq C_k\|\hat{x}_{\{k\}}-y\|,
 \end{equation}
 \item[(iv) ] there is $L_k>0$ such that $ \|F(\hat{x})\hat{x} - F(\hat{x}_{\{k\}})\hat{x}_{\{k\}}\| \leq L_k\|\hat{x}-\hat{x}_{\{k\}}\|$,
 \item[(v) ] the constants satisfy $\beta_k^2\leq 1/\mu_k<\frac{3}{2}\alpha_k^2-4\|\hat{x}_{\{k\}}\|^2C_k^2$,
 \end{itemize}
 then the iterative scheme \eqref{eq:iteration} converges towards $x^{\star}$ satisfying 
 \begin{equation*}
 \|x^\star -  \hat{x}\|\leq \gamma\|e\|+ (1+\gamma c+\gamma L_k) \|\hat{x}-\hat{x}_{\{k\}}\|,
 \end{equation*}
 where  $\gamma=\frac{2}{0.75\alpha^2_k-1/\mu_k-2\|\hat{x}_{\{k\}}\|^2C^2_k}$.
 \end{theorem}
 
 Note that \eqref{eq:RIP again} is again a RIP condition for each $F(z)$. 
 The proof of Theorem \ref{th:blumi} is based on Blumensath's findings in \cite{Blumensath:2012fk}, where the nonlinear operator $A$ is replaced by its first order approximation at $x^{(j)}$ within the iterative scheme. When dealing with the quasi-linear setting, it is natural to use $F(x^{(j)})$, so we formulated the iteration in this way already.
 \begin{proof}
 We first verify that the assumptions in \cite[Corollary 2]{Blumensath:2012fk} are satisfied. 
By using \eqref{eq:new 1 new}, we derive, for $k$-sparse $y\in\R^d$,
\begin{align*}
\| F(\hat{x}_{\{k\}})\hat{x}_{\{k\}} -F(y)y - F(y)(\hat{x}_{\{k\}}-y)\| & =  \|F(\hat{x}_{\{k\}})\hat{x}_{\{k\}}- F(y)\hat{x}_{\{k\}}  \|\\
& \leq\|\hat{x}_{\{k\}}\| C_k \| \hat{x}_{\{k\}}-y \|.
\end{align*}
The assumptions of \cite[Corollary 2]{Blumensath:2012fk} are satisfied, which implies that the iterative scheme \eqref{eq:iteration} converges to some $k$-sparse $x^\star$ satisfying
\begin{equation*}
 \|x^\star -  \hat{x}\|\leq  \gamma \|b-F(\hat{x}_{\{k\}})\hat{x}_{\{k\}}\|+\|\hat{x}_{\{k\}}-\hat{x}\|,
\end{equation*}
where $\gamma=\frac{2}{0.75\alpha^2_k-1/\mu_k-2\|\hat{x}_{\{k\}}\|^2C^2_k}$. We still need to estimate the left term on the right-hand side. A zero addition yields
\begin{align*}
\|b-F(\hat{x}_{\{k\}})\hat{x}_{\{k\}}\| & \leq \|e\|+\|F(\hat{x})\hat{x}-F(\hat{x}_{\{k\}})\hat{x}_{\{k\}}\|\\
& \leq \|e\| + L_k\|\hat{x}-\hat{x}_{\{k\}}\| ,
\end{align*}
which concludes the proof.
\end{proof}

We shall verify that the map in Proposition \ref{th:example} satisfies the assumptions of Theorem \ref{th:blumi} at least when $\hat{x}$ is $k$-sparse:
\begin{example}\label{ex:1}
Let $F(x)=A_1+\epsilon f(\|x-x_0\|)A_2$, so that $F(x)x=A(x)$ with $A$ as in Proposition \ref{th:example}. By a similar proof and under the same notations we derive that an upper bound in \eqref{eq:RIP again} can be chosen as $\beta=1+\delta+\epsilon B \|A_2\|$, where $\delta$ is the RIP constant for $A_1$. For the lower bound, we compute $\alpha=1-\delta-\epsilon B \|A_2\|$. In other words, $F(x)$ satisfies the RIP of order $k$ with constant $\delta+\epsilon B \|A_2\|$. In \eqref{eq:new 1 new}, we can choose $C=\epsilon L \|A_2\|$. Thus, if $\epsilon,B,\|A_2\|,\delta$ are sufficiently small, then the assumptions of Theorem \ref{th:blumi} are satisfied.
\end{example}

\subsection{Iterative soft-thresholding}
The type of thresholding in the scheme \eqref{eq:iteration} of the previous section is one among many potential choices. Here, we shall discuss soft-thresholding, which is widely applied when dealing with linear compressed sensing problems.  
Our findings in this section are based on preliminary results in \cite{Sigl:2013fk}. We suppose that the original signal is sparse and, in fact, we aim to reconstruct the sparsest $\hat{x}$ that matches the data $b=F(\hat{x})\hat{x}$. In other words, we intend to solve for $\hat{x}\in\R^d$ with $\hat{x}=\arg\min \|x\|_{\ell_0}$ subject to $F(x)x=b$. Theorem \ref{th:blumi} yields that we can use iterative hard-thresholding to reconstruct the sparsest solution. Here, we shall follow a slightly different strategy. As $\ell_0$-minimization is a combinatorial optimization problem and computationally cumbersome in principle, even NP-hard under many circumstances, it is common practice in compressed sensing to replace the $\ell_0$-pseudo norm with the $\ell_1$-norm, so that we consider the problem
\begin{equation}\label{eq:min l1}
\arg\min\|x\|_{\ell_1}\quad\text{subject to}\quad F(x)x=b.
\end{equation}
It is also somewhat standard to work with an additional relaxation of it and instead solve for $\hat{x}_\alpha$ given by
\begin{equation}\label{eq:min relaxed}
\hat{x}_\alpha:=\arg\min_{x\in\R^d}\mathcal{J}_\alpha(x),\qquad\text{where}\qquad \mathcal{J}_\alpha(x):=\|F(x)x-b\|^2_{\ell_2} + \alpha \|x\|_{\ell_1},
\end{equation}
where $\alpha>0$ is sometimes called the relaxation parameter. The optimization \eqref{eq:min relaxed} allows for $F(x)x\neq b$, hence, is particularly beneficial when we deal with measurement noise, so that $b=F(\hat{x})\hat{x}+e$ and $\alpha$ can be suitably chosen to compensate for the magnitude of $e$. If there is no noise term, then \eqref{eq:min relaxed} approximates \eqref{eq:min l1} when $\alpha$ tends to zero. The latter is a standard result but we explicitly state this and prove it for the sake of completeness:
\begin{proposition}\label{prop:1}
Let the map $x\mapsto F(x)$ be continuous and suppose that $(\alpha_n)_{n=1}^\infty$ is a sequence of nonnegative numbers that converge towards $0$. If $(\hat{x}_{\alpha_n})_{n=1}^\infty$ is any sequence of minimizers of $\mathcal{J}_{\alpha_n}$, then it contains a subsequence that converges towards a minimizer of \eqref{eq:min l1}. If the minimizer of \eqref{eq:min l1} is unique, then the entire sequence $(\hat{x}_{\alpha_n})_{n=1}^\infty$ converges towards this minimizer. 
\end{proposition}

\begin{proof}
Let $\hat{x}$ be a minimizer of \eqref{eq:min l1}. Direct computations yield
\begin{equation}\label{eq:l1 tmp}
\|\hat{x}_{\alpha_n}\|_{\ell_1}\leq \frac{1}{\alpha_n}\mathcal{J}_{\alpha_n}(\hat{x}_{\alpha_n})\leq \frac{1}{\alpha_n}\mathcal{J}_{\alpha_n}(\hat{x}) = \|\hat{x}\|_{\ell_1}.
\end{equation}
Thus, there is a convergent subsequence $(\hat{x}_{\alpha_{n_j}})_{j=1}^\infty\rightarrow \bar{x}\in\R^d$, for $j\rightarrow\infty$. Since
\begin{align*}
\|F(\bar{x})\bar{x}-y\| = \lim_{j\rightarrow\infty} \|F(\hat{x}_{\alpha_{n_j}})\hat{x}_{\alpha_{n_j}}-y\|\leq \lim_{j\rightarrow\infty} \mathcal{J}_{\alpha_{n_j}}(\hat{x}_{\alpha_{n_j}}) \leq \lim_{j\rightarrow\infty}\mathcal{J}_{\alpha_{n_j}}(\hat{x})=0
\end{align*}
and \eqref{eq:l1 tmp} hold, $\bar{x}$ must be a minimizer of \eqref{eq:min l1}. 

Now, suppose that the minimizer $\hat{x}$ of \eqref{eq:min l1} is unique. If $x_0$ is an accumulation point of $(\hat{x}_{\alpha_{n}})_{j=1}^\infty$, then there is a subsequence converging towards $x_0$. The same arguments as above with the uniqueness assumption yield $x_0=\hat{x}$, so that  $(\hat{x}_{\alpha_{n}})_{j=1}^\infty$ is a bounded sequence with only one single accumulation point. Hence, the entire sequence converges towards $\hat{x}$. 
\end{proof}

From here on, we shall focus on \eqref{eq:min relaxed}, which we aim to solve at least approximately using some iterative scheme. First, we define the map
\begin{equation*}
\mathscr{S}_\alpha:\R^d\rightarrow\R^d,\qquad x\mapsto \mathscr{S}_\alpha(x):=\arg\min_{y\in\R^d} \|F(x)y-b\|^2+\alpha\|y\|_{\ell_1}.
\end{equation*} 
To develop the iterative scheme, we present some conditions, so that $\mathscr{S}_\alpha$ is contractive:
\begin{theorem}\label{th:th elast one maybe}
Given $b\in\R^n$, fix $\alpha>0$ and suppose that there are constants $c_1,c_2,c_3,\gamma>0$ such that, for all $x,y\in\R^d$,

\begin{enumerate}[(i)]
\item \label{it:i}  $ 
\|F(x)\|_2\leq c_1,
$ 
\item \label{it:ii}  there is $z_x\in\R^d$ such that $\|z_x\|_{\ell_1}\leq c_2\|b\|$ and $F(x)z_x=b$, 
\item \label{it:iii}  
$ 
\|F(x)-F(y)\|_2 \leq c_3 \|x-y\|,
$ 
\item \label{it:iv}   if $y$ is $\frac{4}{\alpha^2}(c_1+c_2+c_1^2c_2)^2\|b\|^2$-sparse, then 
\begin{equation*}
(1-\gamma)\|y\|^2 \leq \|F(x)y\|^2\leq (1+\gamma) \|y\|^2,
\end{equation*}
\item \label{it:v}   the constants satisfy $\gamma<1-(1+2c_1c_2)c_3\|b\|$,
\end{enumerate}
then $\mathscr{S}_\alpha$ is a bounded contraction, so that the recursive scheme $x^{(j+1)}_\alpha:=\mathscr{S}_\alpha(x^{(j)}_\alpha)$  converges for any initial vector  towards a point $x_\alpha$ satisfying {  the fixed point relationship}
\begin{equation}\label{eq:fixy}
x_\alpha = \arg\min_{y\in\R^d} \|F(x_\alpha)y-b\|^2+\alpha\|y\|_{\ell_1}.
\end{equation}
\end{theorem}

\begin{remark}
We believe that the fixed point of \eqref{eq:fixy}  in Theorem \ref{th:th elast one maybe} is close to the actual minimizer of \eqref{eq:min relaxed}. To support this point of view, we shall later investigate on the distance $\|x_\alpha-\hat{x}_\alpha\|$ in Theorem \ref{theorem:very last now} and also provide some numerical experiments in Section \ref{sec:soft numerics}. 
\end{remark}

A few more comments are in order before we take care of the proof: 
Note that the constant $c_1$ must hold for the operator norm in \eqref{it:i}, and $\gamma$ in the RIP of \eqref{it:iv} covers only sparse vectors. Therefore, $1+\gamma$ can be much smaller than $c_1$. Condition \eqref{it:iii} is a standard Lipschitz property. If $\gamma$ is indeed less than $1$, then small data $b$ can make up for larger other constants, so that \eqref{it:v} can hold. The requirement \eqref{it:ii} is more delicate though and a rough derivation goes as follows: the data $b$ are supposed to lie in the range of $F(x)$, which is satisfied, for instance, if $F(x)$ is onto. The pseudo-inverse of $F(x)$ then yields a vector $z_x$ with minimal $\ell_2$-norm. We can then ask for boundedness of all operator norms of the pseudo-inverses. However, we still need to bound the $\ell_1$-norm  by using the $\ell_2$-norm, which introduces an additional factor $\sqrt{d}$. 

We introduce the soft-thresholding operator $\mathbb{S}_\alpha:\R^d\rightarrow\R^d$, $x\mapsto \mathbb{S}_\alpha(x)$  given by
{  \begin{equation}\label{eq:soft t}
(\mathbb{S}_\alpha(x))_i=\begin{cases}   
x_i-\alpha/2,&\alpha/2\leq x_i\\
0 ,& -\alpha/2<x_i<\alpha/2\\
x_i+\alpha/2,& x_i \leq -\alpha/2
\end{cases},
\end{equation}
}
which we shall use in the following proof:
\begin{proof}[Proof of Theorem \ref{th:th elast one maybe}]
For $x\in\R^d$, we can apply (ii), so that $F(x)z_x=b$ and $\|z_x\|_{\ell_1}\leq c_2\|b\|$, implying
\begin{align*}
\alpha\|\mathscr{S}_\alpha(x)\| \leq  \|F(x)\mathscr{S}_\alpha(x)-b\|^2+ \alpha\|\mathscr{S}_\alpha(x)\|_{\ell_1} \leq   \alpha\|z_x\|_{\ell_1}
 \leq \alpha c_2 \|b\|.
\end{align*}
Thus, we have $\|\mathscr{S}_\alpha(x)\|\leq c_2 \|b\|$. 

The conditions \eqref{it:i} and \eqref{it:iii} imply
\begin{equation}\label{eq:A and so on}
\|F(x)^*F(x)-F(y)^*F(y)\| \leq 2c_1 c_3 \|x-y\|.
\end{equation}
It is well-known that 
\begin{equation}\label{eq:equality for S}
\mathscr{S}_\alpha(x) = \mathbb{S}_\alpha\big( \xi(x) \big), \qquad\text{where} \qquad \xi(x)= (I-F(x)^*F(x)))\mathscr{S}_\alpha(x)+ F(x)^*b,
\end{equation}
and $\mathbb{S}_\alpha$ is the soft-thresholding operator in \eqref{eq:soft t}, cf.~\cite{Daubechies:2004aa}. Note that $\xi$ can be bounded by
\begin{align*}
\|\xi(x)\| & = \|(I-F(x)^*F(x))\mathscr{S}_\alpha(x)\|+ \|F(x)^*b\|\\
& \leq  c_2\|b\|+c_1^2c_2\|b\| + c_1\|b\|\\
& = (c_1+c_2+c_1^2c_2)\|b\|.
\end{align*}
It is known, cf.~\cite{DeVore:1998aa} and \cite[Lemma 4.15]{:2010fk}, that the bound on $\xi$ implies
\begin{equation*}
\#\supp\big(\mathscr{S}_\alpha(x)\big) = \#\supp\big( \mathbb{S}_\alpha(\xi(x)) \big)\leq \frac{4}{\alpha^2}(c_1+c_2+c_1^2c_2)^2\|b\|^2.
\end{equation*}
The condition \eqref{it:iv} can be rewritten as $\| (I - F(x)^*F(x))y\|\leq \gamma \|y\|$, for suitably sparse $y$. Since soft-thresholding is nonexpansive, i.e., 
\begin{equation*}
\|\mathbb{S}_\alpha(x)-\mathbb{S}_\alpha(y)\|\leq \|x-y\|, 
\end{equation*}
the identity \eqref{eq:equality for S} then yields with \eqref{eq:A and so on}
\begin{align*}
\|\mathscr{S}_\alpha(x)-\mathscr{S}_\alpha(y)\|& \leq \|\mathscr{S}_\alpha(x)-\mathscr{S}_\alpha(y)+F(x)^*b-F(y)^*b \\
&\qquad - F(x)^*F(x)\mathscr{S}_\alpha(x)+F(y)^*F(y)\mathscr{S}_\alpha(y)\|\\
& \leq  \|(I-F(x)^*F(x))(\mathscr{S}_\alpha(x)-\mathscr{S}_\alpha(y))\| \\
&\qquad   +\|(F(y)^*F(y)-F(x)^*F(x))\mathscr{S}_\alpha(y)\|    + \|(F(x)^*-F(y)^*)b\|\\
& \leq \gamma\|\mathscr{S}_\alpha(x)-\mathscr{S}_\alpha(y)\| + 2c_1c_2c_3\|b\| \|x-y\|+ c_3\|x-y\|\|b\|,
\end{align*}
which implies
\begin{equation*}
\|\mathscr{S}_\alpha(x)-\mathscr{S}_\alpha(y)\| \leq  \frac{2c_1c_2+1}{1-\gamma}c_3\|b\| \|x-y\|.
\end{equation*}
Thus, $\mathscr{S}_\alpha$ is contractive and $x^{(j)}_\alpha$ converges towards a fixed point. 
\end{proof}

If $\|b\|$ is large, then the conditions in Theorem \ref{th:th elast one maybe} are extremely strong. For smaller $\|b\|$, on the other hand, we can find examples matching the requirements. Note also that the above proof reveals that $x_\alpha$ is at most {  $\lceil\frac{4}{\alpha^2}(c_1+c_2c_4)^2\|b\|^2 \rceil$}-sparse. 
\begin{example}\label{ex:2}
As in Example \ref{ex:1}, let $F(x)=A_1+\epsilon f(\|x-x_0\|)A_2$, so that $F(x)x=A(x)$ with $A$ as in Proposition \ref{th:example}. We additionally suppose that $n\leq d$, that $A_1$ is onto, and denote the smallest eigenvalue of {  $A_1 A_1^*$} by $\beta>0$. Then we can choose $c_1=\|A_1\|+\epsilon B \|A_2\|_2$ and $c_3=\epsilon L \|A_2\|$. If $\epsilon,B,\|A_2\|$ are sufficiently small, then the smallest eigenvalue of {  $F(x) F(x)^*$} is almost given by the smallest eigenvalue of {  $A_1 A_1^*$} and denoted by $\beta>0$. If {  $F(x)^\dagger = F(x)^* (F(x) F(x)^*)^{-1}$} denotes the pseudo-inverse of $F(x)$, then {  $\|F(x)^\dagger\|_2\leq c_1/\beta$}. For \eqref{it:ii}, we can define {  $y:=F(x)^\dagger b$}, so that $\|y\|_{\ell_1}\leq \sqrt{d}\|y\|\leq \sqrt{d} \frac{c_1}{\beta}\|b\|$ and $c_2\leq \sqrt{d} \frac{c_1}{\beta}$. Still, suppose that $\epsilon,B,\|A_2\|$ are sufficiently small and also assume that $A_1$ satisfies the RIP with constant $0<\delta<1$ for sufficiently large sparsity requirements in \eqref{it:iv}. Thus, the conditions of Theorem \ref{th:th elast one maybe} are satisfied if $\|b\|$ is sufficiently small. 
%
%
\end{example}

\begin{algorithm}
\KwSty{Quasi-linear iterative soft-thresholding}:\\
\KwIn{$F:\R^d\rightarrow \R^{n\times d}$, $b\in\R^n$}
Initialize $x^{(0)}$ as an arbitrary vector 
\\
\For{$j=0,1,2,\ldots$ until some stopping criterion is met}{
\vspace{1ex}
\begin{equation*}
\hspace{-5ex}
x^{(j+1)}_\alpha:=\arg\min_{x\in\R^d}\mathcal{J}^S_\alpha(x,x^{(j)}_\alpha)=\mathbb{S}_\alpha \big( (I- F(x^{(j)}_\alpha)^* F(x^{(j)}_\alpha)) x^{(j)}_\alpha  + F(x^{(j)}_\alpha)^*b \big)
\end{equation*}
}
\KwOut{$x^{(1)}_\alpha$, $x^{(2)}_\alpha$, $\ldots$}
\caption{We propose to iteratively minimize the surrogate functional, which yields a simple iterative soft-thresholding scheme.}\label{algo:2}
\end{algorithm}

The recursive scheme in Theorem \ref{th:th elast one maybe} involving $\mathscr{S}_\alpha$ requires a minimization in each iteration step. To derive a more efficient scheme, we consider the surrogate functional
\begin{equation*}
\mathcal{J}^S_\alpha(x,a)=\|F(a)x-b\|^2+\alpha\|x\|_{\ell_1}+\|x-a\|^2-\|F(a)x-F(a)a\|^2.
\end{equation*}
We have $\mathcal{J}^S_\alpha(x,x)=\mathcal{J}_\alpha(x)$ and propose the iterative Algorithm \ref{algo:2}. In each iteration step, we minimize the surrogate functional {  in the first variable having the second one fixed with the previous iteration}, which only requires a simple soft-thresholding. 

Indeed, iterative soft-thresholding converges towards the fixed point $x_\alpha$:
\begin{theorem}\label{th:very last}
Suppose that the assumptions of Theorem \ref{th:th elast one maybe} are satisfied and let $x_\alpha$ be the $k$-sparse fixed point in \eqref{eq:fixy}.  We define $\hat{z}_\alpha := (I-F(x_\alpha)^*F(x_\alpha))x_\alpha) + F(x_\alpha)^*b$ and $K=\frac{4\|x_\alpha\|^2}{\alpha^2} + \frac{4c}{\alpha}C$, where $C=\sup_{1\leq l<d}(\sqrt{l+1}\|\hat{z}_\alpha-(\hat{z}_\alpha)_{\{l\}}\|_{\ell_2})$ and $c>0$ sufficiently large. Additionally assume that 
\begin{enumerate}[(a)]
\item \label{item:a} there is $0<\tilde{\gamma}<\gamma$ such that, for all $K+k$-sparse vectors $y\in\R^d$, 
\begin{equation}
(1-\tilde{\gamma})\|y\|^2\|F(x_\alpha)y\|^2\leq (1+\tilde{\gamma})\|y\|^2, \label{lastRIP}
\end{equation}
\item \label{item:b}  the constants satisfy $\tilde{\gamma}+(1+4c_1c_2)c_3\|b\|<\gamma$.
\end{enumerate}
Then by using $x^{(0)}_\alpha=0$ as initial vector, the iterative Algorithm \ref{algo:2} converges towards $x_\alpha$ with 
\begin{equation*}
 \|x^{(j)}_\alpha-x_\alpha\|\leq \gamma^j \|x_\alpha\|,\quad j=0,1,2,\ldots.
\end{equation*}
\end{theorem}

Note that the above $k$ is at most {  $\lceil \frac{4}{\alpha^2}(c_1+c_2+c_1^2c_2)^2\|b\|^2 \rceil$}. Also, it may be possible to choose $\gamma$ a little bigger than necessary to ensure $\tilde{\gamma}<\gamma$. Condition (b) can then be satisfied when the magnitude of the data $b$ is sufficiently small. Moreover, if constants are suitably chosen, Example \ref{ex:2} also provides a map $F$ that satisfies the assumptions of Theorem \ref{th:very last} when $\|b\|$ is small.
\begin{proof}
We use induction and observe that the case $j=0$ is {  trivially verified}. Next, we suppose that $x^{(j)}_\alpha$ satisfies $\|x^{(j)}_\alpha-x_\alpha\|\leq \gamma^j \|x_\alpha\|$ and that it has at most $K$ nonzero entries. Our aim is now to verify that $x^{(j+1)}$ also satisfies the support condition and $\|x^{(j+1)}_\alpha-x_\alpha\|\leq \gamma^{j+1} \|x_\alpha\|$. To simplify notation let 
\begin{equation}\label{eq:f}
f(x,y):=(I-F(x)^*F(x))y+F(x)^*b,
\end{equation}
so that $\hat{z}_\alpha=f(x_\alpha,x_\alpha)$. It will be useful later to derive bounds for both terms $\|f(x_\alpha,x_\alpha)-f(x_\alpha,x^{(j)}_\alpha)\|$ and $\|f(x^{(j)}_\alpha,x^{(j)}_\alpha)-f(x^{(j)}_\alpha,x^{(j)}_\alpha)\|$. Therefore, we start to estimate
\begin{equation}\label{eq:I}
\|f(x_\alpha,x_\alpha)-f(x_\alpha,x^{(j)}_\alpha)\|  = \|(I-F(x_\alpha)^*F(x_\alpha))(x_\alpha-x^{(j)}_\alpha)\|
 \leq \tilde{\gamma} \gamma^j\|x_\alpha\|,
\end{equation}
where we have used \eqref{item:a}  in the form $\| (I - F(x_\alpha)^*F(x_\alpha))y\|\leq \tilde{\gamma} \|y\|$ and the induction hypothesis. 

Next, we take care of $\|f(x_\alpha,x^{(j)}_\alpha)-f(x^{(j)}_\alpha,x^{(j)}_\alpha)\|$ and derive
\begin{align*}
\|f(x_\alpha,x^{(j)}_\alpha)-f(x^{(j)}_\alpha,x^{(j)}_\alpha)\| & \leq \|(F(x_\alpha)^*-F(x^{(j)}_\alpha)^*)b\|
\\
&\qquad +\|(F(x^{(j)}_\alpha)^*F(x^{(j)}_\alpha)-F(x_\alpha)^*F(x_\alpha))x^{(j)}_\alpha\|\\
& \leq c_3\gamma^j\|x_\alpha\|\|b\|+ \|(F(x^{(j)}_\alpha)^*F(x^{(j)}_\alpha)-F(x^{(j)}_\alpha)^*F(x_\alpha))x^{(j)}_\alpha\|\\
& \qquad +\|(F(x^{(j)}_\alpha)^*F(x_\alpha)-F(x_\alpha)^*F(x_\alpha))x^{(j)}_\alpha\|\\
& \leq c_3\gamma^j\|x_\alpha\|\|b\|+2c_1c_3 \gamma^j \|x_\alpha\|  \|x^{(j)}_\alpha\|\\
& = \gamma^j\|x_\alpha\| c_3\big( \|b\|+2c_1  \|x^{(j)}_\alpha\|\big).
\end{align*}

By using \eqref{it:ii} and the minimizing property of $x_\alpha$, we derive
\begin{equation*}
 \|x_\alpha\| \leq  \|x_\alpha\|_{\ell_1} \leq \|z_{x_\alpha}\|_{\ell_1} \leq c_2 \|b\|.
\end{equation*}
%
The triangular inequality then yields $\|x^{(j)}_\alpha\|\leq  \gamma^j c_2\|b\|+c_2\|b\|$. Thus, we obtain the estimate
\begin{equation}\label{eq:II}
\|f(x_\alpha,x^{(j)}_\alpha)-f(x^{(j)}_\alpha,x^{(j)}_\alpha)\|\leq \gamma^j\|x_\alpha\| c_3\|b\|\big(1+2c_1 c_2(1+ \gamma^j )\big).
\end{equation}

According to \eqref{eq:I} and \eqref{eq:II}, the condition \eqref{item:b} yields 
\begin{equation}\label{eq:same way}
\|\hat{z}_\alpha-f(x^{(j)}_\alpha,x^{(j)}_\alpha)\| \leq \gamma^j\|x_\alpha\|( \tilde{\gamma}+c_3\|b\|(1+2c_1 c_2(1+\gamma^j  )))
 \leq \gamma^{j+1}\|x_\alpha\|.
\end{equation}
Results in \cite[Lemma 4.15]{:2010fk} with $x^{(j+1)}_\alpha=\mathbb{S}_\alpha(f(x^{(j)}_\alpha,x^{(j)}_\alpha))$ imply that there is a constant $c>0$ such that
\begin{equation*}
\#\supp(x^{(j+1)}_\alpha) \leq \frac{4\gamma^{2j+2}\|x_\alpha\|^2}{\alpha^2} + \frac{4c}{\alpha}C,
\end{equation*}
where $C=\sup_{1\leq l<d}(\sqrt{l+1}\|z_\alpha-(z_\alpha)_{\{l\}}\|_{\ell_2})$. Since the above right-hand side is smaller than $K$, we have the desired support estimate.

%
%

Next, we take care of the error bounds. 
Since $x_\alpha$ is a fixed point of \eqref{eq:fixy}, we have $x_\alpha=\mathbb{S}_\alpha(\hat{z}_\alpha)$, which we have already used in \eqref{eq:equality for S}. The nonexpansiveness of $\mathbb{S}_\alpha$ yields with $\hat{z}_\alpha=f(x_\alpha,x_\alpha)$
\begin{align*}
\|x_\alpha-x^{(j+1)}_\alpha\| & \leq \|\mathbb{S}_\alpha(f(x_\alpha,x_\alpha)) - \mathbb{S}_\alpha(f(x_\alpha,x^{(j)}_\alpha))\|+\| \mathbb{S}_\alpha(f(x_\alpha,x^{(j)}_\alpha))-\mathbb{S}_\alpha(f(x^{(j)}_\alpha,x^{(j)}_\alpha))\|\\
& \leq \|f(x_\alpha,x_\alpha) -f(x_\alpha,x^{(j)}_\alpha)\|+\| f(x_\alpha,x^{(j)}_\alpha)-f(x^{(j)}_\alpha,x^{(j)}_\alpha)\|.
\end{align*}
The same way as for \eqref{eq:same way}, we use the bounds in \eqref{eq:I} and  \eqref{eq:II} with \eqref{item:b} to derive
\begin{equation*}
\|x_\alpha-x^{(j+1)}_\alpha\| \leq \gamma^j \|x_\alpha\|  ( \tilde{\gamma}+c_3\|b\|(1+4c_1 c_2))
 \leq \gamma^{j+1} \|x_\alpha\|,
 \end{equation*}
so that we can conclude the proof.
%
\end{proof}

It remains to verify that the output $x_\alpha$ of the iterative soft-thresholding scheme is close to the minimizer $\hat{x}_\alpha$ of \eqref{eq:min relaxed}:
\begin{theorem}\label{theorem:very last now}
Suppose that the assumptions of Theorem \ref{th:very last} hold, that there is  a $K$-sparse minimizer $\hat{x}_\alpha$ of \eqref{eq:min relaxed}, and that $\frac{c_2c_3}{\sqrt{1-\tilde{\gamma}}}\|b\|<1$ holds, then we have
\begin{equation*}
\|x_\alpha - \hat{x}_\alpha\| \leq \frac{\sqrt{\alpha c_2\|b\|}}{a}  + \frac{c_1+c_3\|\hat{x}\|}{a}\|\hat{x}_\alpha -\hat{x}\| ,
\end{equation*}
where $\hat{x}$ satisfies $F(\hat{x})\hat{x}=b$ and $a=\sqrt{1-\tilde{\gamma}}-c_2c_3\|b\|$.
\end{theorem}

Note that Proposition \ref{prop:1} yields that the minimizer $\hat{x}_\alpha$ can approximate $\hat{x}$, so that $\|\hat{x}_\alpha -\hat{x}\|$ can become small and, hence,  $\|x_\alpha - \hat{x}_\alpha\|$ must be small. It should be mentioned though that the assumptions of Theorem \ref{th:very last} depend on $\alpha$ because its magnitude steers the sparsity of $x_\alpha$. Therefore, letting $\alpha$ tend to zero  is quite delicate because the assumptions become stronger. Indeed, taking the limit requires that condition \eqref{it:iv} in Theorem \ref{th:th elast one maybe} holds for all $y\in\R^d$, not just for sparse vectors, and the same is required for condition \eqref{item:a} in Theorem \ref{th:very last}. 
\begin{proof}[Proof of Theorem \ref{theorem:very last now}]
We first bound $\hat{x}_\alpha$ by
\begin{align*}
\alpha \|\hat{x}_\alpha\|  \leq  \|F(\hat{x}_\alpha)\hat{x}_\alpha-b\|^2+  \alpha \|\hat{x}_\alpha\|_{\ell_1}
 \leq \|F({x}_\alpha){x}_\alpha-b\|^2+  \alpha \|{x}_\alpha\|
 \leq \alpha \|z_{x_\alpha}\|_{\ell_1}.
\end{align*}
Therefore, we have $\|\hat{x}_\alpha\|\leq c_2\|b\|$. Since $\hat{x}_\alpha$ is $K$-sparse, we derive
\begin{align*}
\|F(x_\alpha)x_\alpha - F(\hat{x}_\alpha)\hat{x}_\alpha\| &\geq \|F(x_\alpha)x_\alpha - F(x_\alpha)\hat{x}_\alpha\| - \|F(x_\alpha)\hat{x}_\alpha - F(\hat{x}_\alpha)\hat{x}_\alpha\| \\
& \geq \sqrt{1-\tilde{\gamma}} \|x_\alpha-\hat{x}_\alpha\| -c_3\|x_\alpha-\hat{x}_\alpha\|\|\hat{x}_\alpha\|.
\end{align*}
These computations and a zero addition imply
\begin{align*}
\|x_\alpha - \hat{x}_\alpha\| & \leq \frac{1}{a}\|F(x_\alpha)x_\alpha-F(\hat{x}_\alpha)\hat{x}_\alpha\|\\
&\leq \frac{1}{a} (\|F(x_\alpha)x_\alpha-b\|+\|F(\hat{x}_\alpha)\hat{x}_\alpha-b\|).
\end{align*}
We shall now bound both terms on the right-hand side separately. 
The minimizing property and \eqref{it:ii} in Theorem \ref{th:th elast one maybe} yield
\begin{align*}
\|F(x_\alpha)x_\alpha-b\|^2  \leq  \alpha \|z_{x_\alpha}\|_{\ell_1}-\alpha\|x_\alpha\|_{\ell_1}
 \leq \alpha c_2\|b\|.
\end{align*}
The second term is bounded by
\begin{align*}
\|F(\hat{x}_\alpha)\hat{x}_\alpha-b\| & \leq \|F(\hat{x}_\alpha)\hat{x}_\alpha-F(\hat{x}_\alpha)\hat{x}\|+\|F(\hat{x}_\alpha)\hat{x}-F(\hat{x})\hat{x}\|\\
& \leq c_1 \|\hat{x}_\alpha -\hat{x}\|+c_3\|\hat{x}\|\|\hat{x}_\alpha -\hat{x}\|\\
& = (c_1+c_3\|\hat{x}\|)\|\hat{x}_\alpha -\hat{x}\|,
\end{align*}
so that we can conclude the proof. 
\end{proof}

Alternatively, we can also bound the distance between $x_\alpha$ and $\hat{x}$:
\begin{proposition}
Suppose that the assumptions of Theorem \ref{th:very last} hold, that we can replace $x_\alpha$ in condition (a) of the latter theorem with some $K$-sparse $\hat{x}\in\R^d$ satisfying $F(\hat{x})\hat{x}=b$, and that $\frac{c_2c_3}{\sqrt{1-\tilde{\gamma}}}\|b\|<1$ holds, then we have
\begin{equation*}
\|x_\alpha - \hat{x}\| \leq \frac{\sqrt{\alpha c_2\|b\|}}{\sqrt{1-\tilde{\gamma}}-c_2 c_3\|b\|}.
\end{equation*}
\end{proposition}
\begin{proof}
We can estimate
\begin{align*}
\|x_\alpha - \hat{x}\| & \leq \frac{1}{\sqrt{1-\tilde{\gamma}}} \|F(\hat{x}) x_\alpha - F(\hat{x})\hat{x}  \|\\
& \leq  \frac{1}{\sqrt{1-\tilde{\gamma}}} \big( \|F(\hat{x}) x_\alpha - F(x_\alpha)x_\alpha\| + \|F(x_\alpha)x_\alpha-F(\hat{x})\hat{x}  \|\big)\\
& \leq  \frac{1}{\sqrt{1-\tilde{\gamma}}} \big( c_3\|x_\alpha\|\|x_\alpha-\hat{x}\| +  \sqrt{\|F(x_\alpha)x_\alpha-b\|^2+\alpha\|x_\alpha\|_{\ell_1}}  \big)\\
&  \leq  \frac{1}{\sqrt{1-\tilde{\gamma}}} \big( c_3c_2\|b\| \|x_\alpha-\hat{x}\|+  \sqrt{\alpha\|z_{x_\alpha}\|_{\ell_1}}  \big)\\
&  \leq  \frac{1}{\sqrt{1-\tilde{\gamma}}} \big( c_3c_2\|b\| \|x_\alpha-\hat{x}\|+  \sqrt{\alpha c_2\|b\|}  \big).
\end{align*}
From here on, some straight-forward calculations yield the required statement.
\end{proof}

\subsection{Numerical experiments for iterative thresholding}\label{sec:soft numerics}

Theorem \ref{th:very last} provides a simple thresholding algorithm to compute the fixed point $x_\alpha$ in \eqref{eq:fixy} that is more efficient than the recursive scheme in Theorem \ref{th:th elast one maybe}. To support Theorem \ref{theorem:very last now}, we shall check numerically that $x_\alpha$ is indeed close to a minimizer $\hat{x}_\alpha$ of \eqref{eq:min relaxed}.

The quasi-linear measurements are taken from the Examples \ref{ex:1} and \ref{ex:2}. The recovery rates from iterative hard- and soft-thresholding are plotted in Figure \ref{sub:2} and show a phase transition. For soft-thresholding, this transition depends on both, the sparsity level $k$ and the measurement magnitude. Those observations are consistent with the theoretical results in  Theorem \ref{th:very last} und suggest that the original signal can be recovered by iterative soft-thresholding. We also use hard-thresholding but for comparable parameter choices the signal was only recovered when $k=1$, cf.~Fig.~\ref{sub:1}.
\begin{figure}
\centering
\subfigure[iterative soft-thresholding]{
\frame{\includegraphics[width=.3\textwidth]{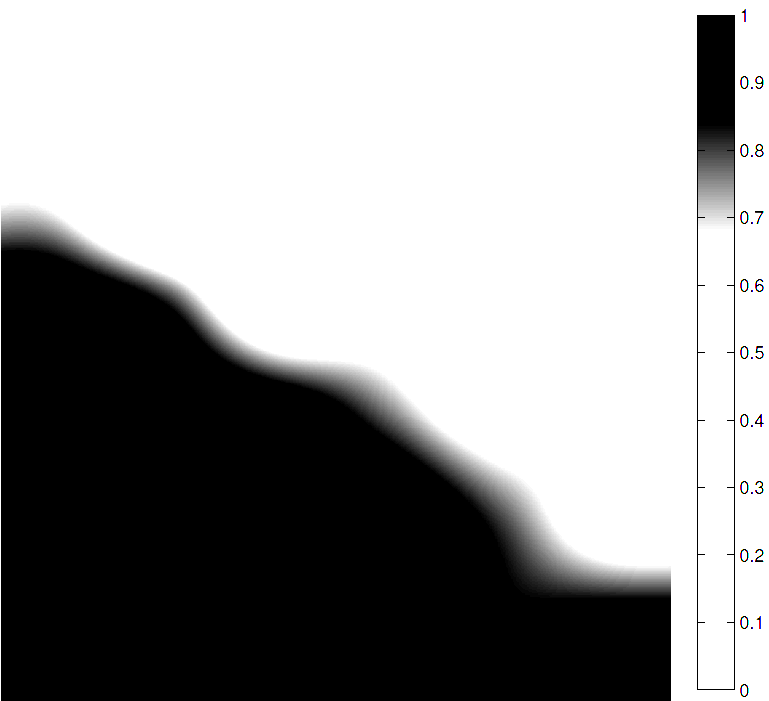}\label{sub:2}}}
\subfigure[iterative hard-thresholding]{
\frame{\includegraphics[width=.3\textwidth]{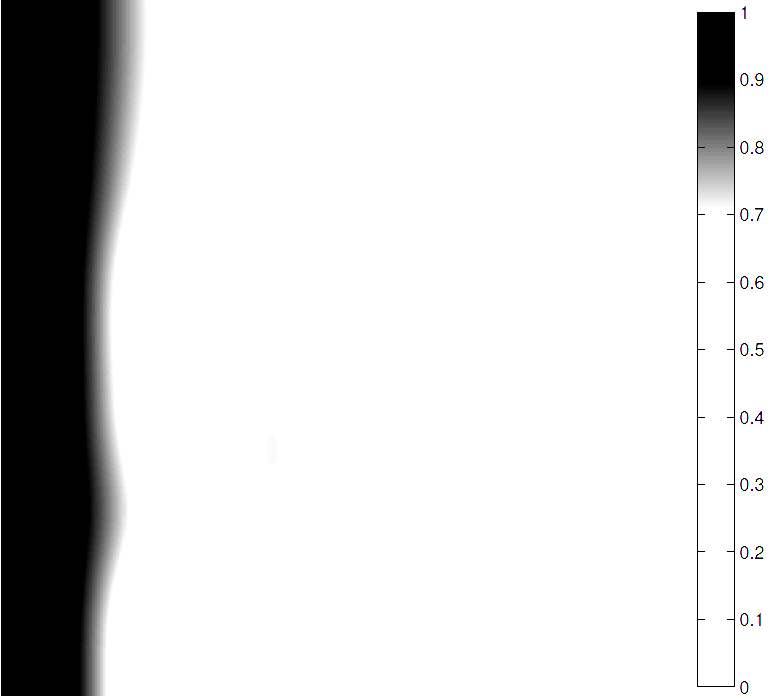}}\label{sub:1}}
\caption{Recovery rates for iterative hard- and soft-thresholding used with the measurements in the Examples \ref{ex:1} and \ref{ex:2} with $d=80$, $n=20$, $A_1$ having i.i.d.~Gaussian entries, $\epsilon=1$, and $A_2=I$. The sparsity parameter $k$ runs on the horizontal axis from $1$ to $10$, the norm of $\hat{x}$ runs on the vertical axis from $0.01$ to $1$. As expected, the recovery rates decrease with growing $k$.  Consistent with the theory, we also observe decreased recovery rates for larger signal norms with soft-thresholding. Hard-thresholding appears only successful for these parameters when $k=1$, but throughout the entire range of considered signal norms. 
}\label{fig:hard soft}
\end{figure}

The assumptions of the Theorems \ref{th:blumi} and \ref{th:very last} cannot be satisfied within the phase retrieval setting, and the initial vectors $x^{(0)}=x^{(0)}_\alpha=0$ in \eqref{eq:iteration} and in Algorithm \ref{algo:2}, respectively, would lead to a sequence of zero vectors. We observed numerically, that other choices of initial vectors do not yield {  acceptable} recovery rates either, so that we did not pursue this direction. 

  \section*{Acknowledgments} 
{  Martin Ehler~acknowledges the financial support by the Vienna Science and Technology Fund (WWTF) through project VRG12-009 and the Research Career Transition Awards Program EH 405/1-1/575910 of the National Institutes of Health and the German Science Foundation. Massimo Fornasier~is supported by the ERC-Starting Grant for the project ``High-Dimensional Sparse Optimal Control''. Juliane Sigl acknowledges the partial financial support of the START-Project ``Sparse Approximation and Optimization in High-Dimensions'' and the hospitality of the Johann Radon Institute for Computational and Applied Mathematics, Austrian Academy of Sciences, Linz, during the early preparation of this work.}


\end{document}